\documentclass[oneside,english]{amsart}
\usepackage[T1]{fontenc}
\usepackage[latin9]{inputenc}
\usepackage{amsthm}
\usepackage{amssymb}
\usepackage{wasysym}
\usepackage{esint}
\usepackage[all]{xy}

\makeatletter
\numberwithin{figure}{section}
\theoremstyle{plain}
\newtheorem{thm}{\protect\theoremname}
  \theoremstyle{definition}
  \newtheorem{defn}[thm]{\protect\definitionname}
  \theoremstyle{remark}
  \newtheorem{rem}[thm]{\protect\remarkname}
  \theoremstyle{definition}
  \newtheorem{example}[thm]{\protect\examplename}
  \theoremstyle{plain}
  \newtheorem{prop}[thm]{\protect\propositionname}
  \theoremstyle{plain}
  \newtheorem{cor}[thm]{\protect\corollaryname}
  \theoremstyle{plain}
  \newtheorem{lem}[thm]{\protect\lemmaname}

\usepackage{amsfonts}\usepackage{amsthm}\usepackage{enumerate}\usepackage{mathrsfs}\usepackage{cancel}\usepackage{url}
\usepackage{hyperref}
\usepackage{bbm}

\input xy\huge
\xyoption{all}

\DeclareFontFamily{OT1}{pzc}{}
\DeclareFontShape{OT1}{pzc}{m}{it}{<-> s * [1.10] pzcmi7t}{}
\DeclareMathAlphabet{\mathpzc}{OT1}{pzc}{m}{it}

\newcommand{\rr}{\mathbb{R}}

\newcommand{\pp}{\mathbb{P}}

\newcommand{\ff}{\mathcal{F}}

\newcommand{\kk}{\mathcal{K}}

\newcommand{\cc}{\mathbb{C}}

\newcommand{\zz}{\mathbb{Z}}

\newcommand{\lc}{\mathcal{L}}

\newcommand{\bb}{\mathcal{B}}

\newcommand{\im}{\mbox{Im}\thinspace}

\newcommand{\cl}{\mathcal{C}}

\newcommand{\bl}{\text{Bl}\thinspace}

\newcommand{\hotimes}{\hat{\otimes}}
\newcommand{\fgvect}{\textbf{FGVect}}
\newcommand{\hfgvect}{\widehat{\fgvect}}
\newcommand{\grading}{(\zz \oplus \zz/2\zz)}

\newcommand{\codim}{\text{cd}\thinspace}
\newcommand{\ls}{\text{ls}\thinspace}

\newcommand{\bilam}{\textbf{BiMod}_\Lambda}
\newcommand{\hbilam}{\widehat{\bilam}}

\newcommand{\tb}{\mathbb{T}}
\newcommand{\ca}{\cc[\vec{\alpha}]}
\newcommand{\mf}{\mathfrak{m}}
\newcommand{\intcone}{\zz_{\geq 0}^{n'}}
\newcommand{\bva}{[\vec{\alpha}]}
\newcommand\mapsfrom{\mathrel{\reflectbox{\ensuremath{\mapsto}}}}
\usepackage{tikz} 
\usetikzlibrary{matrix,arrows}

\makeatother

\usepackage{babel}
  \providecommand{\corollaryname}{Corollary}
  \providecommand{\definitionname}{Definition}
  \providecommand{\examplename}{Example}
  \providecommand{\lemmaname}{Lemma}
  \providecommand{\propositionname}{Proposition}
  \providecommand{\remarkname}{Remark}
\providecommand{\theoremname}{Theorem}

\begin{document}

\title{Equivariant $A_{\infty}$ Algebras for Nonorientable Lagrangians}

\author{Amitai Zernik}

\maketitle

\begin{abstract}
We set up an algebraic framework for the study of pseudoholomorphic
discs bounding nonorientable Lagrangians, as well as equivariant extensions
of such structures arising from a torus action.

First, we define unital cyclic twisted $A_{\infty}$ algebras and
prove some basic results about them, including a homological perturbation
lemma which allows one to construct minimal models of such algebras.
We then construct an equivariant extension of $A_{\infty}$ algebras
which are invariant under a torus action on the underlying complex.
Finally, we construct a homotopy retraction of the Cartan-Weil complex
to equivariant cohomology, which allows us to construct minimal models
for equivariant cyclic twisted $A_{\infty}$ algebras.

In a forthcoming paper we will use these results to define and obtain
fixed point expressions for the open Gromov-Witten theory of $\rr\pp^{2n'}\hookrightarrow\cc\pp^{2n'}$,
as well as its equivariant extension.
\end{abstract}
\tableofcontents{}

\section{Introduction}

This paper lays the algebraic foundations for our definition of the
equivariant open Gromov-Witten invariants for $\rr\pp^{2n'}$, and
their computation using $A_{\infty}$ fixed-point localization. We
briefly describe the geometric motivation, before outlining the main
results. 

Let $\left(M,\omega\right)$ be a compact symplectic manifold and
$L\overset{i}{\hookrightarrow}M$ be an embedded Lagrangian submanifold.
Such an embedding is\emph{ }called \emph{relatively spin }if $L$
is oriented and $w_{2}\left(TL\right)\in\mbox{\ensuremath{\mathrm{Im}}}\,\left(i^{*}\right)\subset H^{2}\left(L;\zz/2\zz\right)$
(cf. Definition 3.1.1 in \cite{FOOO}). Fukaya \cite{fukayacyclic}
associates a $G$-gapped cyclic filtered $A_{\infty}$ algebra to
every such relatively spin Lagrangian embedding. This is a cyclic
deformation of the differential graded algebra (DGA) of differential
forms on $L$, obtained by taking into account \emph{quantum corrections},
namely, the effects of positive-energy pseudoholomorphic discs in
$M$ bounded by $L$.

The Lagrangian embedding $\rr\pp^{2n'}\hookrightarrow\cc\pp^{2n'}$
is not relatively spin. Indeed, $\rr\pp^{2n'}$ is not even orientable.
On the other hand, Solomon \cite{JT} showed that for $n'=1$ the
moduli spaces associated with pseudoholomorphic discs bounding $\rr\pp^{2}$
can be used to define invariants which are equivalent to Welschinger's
\cite{welschinger} signed counts of real rational planar curves.
One upshot of extending the $A_{\infty}$ formalism to accommodate
non-orientable $L$ such as $\rr\pp^{2n'}$ is that it allows to generalize
the definition of the Solomon-Welschinger invariants to all $n'$.
Roughly speaking, this is because it allows to ``keep tabs'' of
boundary corrections, which for $n'=1$ happen to vanish but in higher
dimensions become significant. This approach for defining invariants
is based on \cite{jake+sara}.

Another feature of the embedding $\rr\pp^{2n'}\hookrightarrow\cc\pp^{2n'}$
is that it is equivariant with respect to an action of the rank $n'$
torus group $\tb$. This motivates one to look for $H^{\bullet}\left(B\tb\right)$-valued
invariants generalizing the classical invariants discussed in the
previous paragraph, and compute them using fixed-point localization.
Here the boundary contributions become significant already for $n'=1$. 

With this in mind, our goal in this paper is to set up an $A_{\infty}$
formalism that will (i) capture the quantum deformations associated
with non-orientable Lagrangians, and (ii) handle equivariant extensions
of $\tb$-invariant $A_{\infty}$ algebras. Let us now explain how
this is carried out.

\sloppy A non-orientable Lagrangian embedding $L\overset{i}{\hookrightarrow}M$
will be called \emph{relatively $Pin^{-}$} if ${w_{2}\left(TL\right)+w_{1}\left(TL\right)^{2}\in\im\left(i^{*}\right)}$
(see \cite{JT}). In this case, the Maslov index of holomorphic discs
may be odd, and one must allow for forms with values in the orientation
local system, which leads to some subtle signs in the computations.
In section \ref{sec:Twisted-A8-algebras} we introduce a generalization
of the notion of a $G$-gapped filtered $A_{\infty}$ algebra which
we call a \emph{twisted $A_{\infty}$ algebra} (see Definition \ref{def:twisted A8 algebra}),
that captures this situation. More precisely, in a forthcoming paper
we will construct a twisted $A_{\infty}$ algebra for the $Pin^{-}$
Lagrangian embedding $\rr\pp^{2m}\hookrightarrow\cc\pp^{2m}$. It
should be possible to construct such an algebra for \emph{any }relatively
$Pin^{-}$ Lagrangian embedding. As usual the \emph{easy-to-check}
Definition \ref{def:twisted A8 algebra} is followed by an equivalent
\emph{easy-to-use }definition of twisted $A_{\infty}$ algebras, as
tame differentials on a certain bar coalgebra, see Proposition \ref{prop:Differentials are twisted A8 algebras}.
Cyclic and unital versions are also discussed.

In section \ref{sec:Homological-algebra} we prove the homological
perturbation lemma for twisted $A_{\infty}$ algebras and discuss
cyclic and unital versions too, see Theorem \ref{thm:homological perturbation lemma}
and Proposition \ref{prop:unital and cyclic HPL}. This is an important
computational tool, which enables one to construct minimal models.
To apply these results a certain retraction is needed, see Definitions
\ref{def:retraction} and \ref{def:cyclic unital retraction}. 

The remainder of the paper is devoted to discussing the equivariant
situation. When a manifold $L$ is equipped with an action of the
rank $n'$ torus group $\tb$, one can construct the \emph{Cartan-Weil
DGA} which is a certain extension of the De Rham $\cc$-DGA of $L$
defined over $H^{\bullet}\left(B\tb\right)=\cc\left[\alpha_{1},...,\alpha_{n'}\right]$.
In Section \ref{sec:Equivariant Cohomology} we give a self-contained
and fairly thorough account of this, including a discussion of the
equivariant angular form (see Definition \ref{def:equivariant angular form}
and Proposition \ref{prop:equiv ang form exists}) and Poincare duality,
Corollary \ref{cor:equivariant poincare duality}.

In Section \ref{sec:Equivariant A8 algebras} we apply the Cartan-Weil
theory to twisted $A_{\infty}$ algebras. We define what it means
for a twisted $A_{\infty}$ algebra to be invariant under an action
of a torus group (Definition \ref{def:invariant twisted A8-1}) and
show that in this case the twisted $A_{\infty}$ algebra admits an
equivariant extension, see Proposition \ref{prop:equivariant twisted extension-1}.
Since DGA's are special cases of twisted $A_{\infty}$ algebras (see
Example \ref{exa:DGA}), we can summarize the situation with the following
commutative square of differential coalgebras.
\[
\xymatrix{\left(\bb_{\cc},d+\wedge\right) & \left(\bb_{\cc\left[\vec{\alpha}\right]}^{CW},D+\wedge\right)\ar[l]_{\vec{\alpha}=0}\\
\left(\bb_{\Lambda_{0}^{G}\left(\cc\right)},m\right)\ar[u]^{\tau=0} & \left(\bb_{\Lambda_{0}^{G}\left(\cc\left[\vec{\alpha}\right]\right)}^{CW},m^{CW}\right)\ar[l]_{\vec{\alpha}=0}\ar[u]^{\tau=0}
}
\]
$\left(\bb_{\cc},d+\wedge\right)$ is the differential bar coalgebra
over $\cc$, corresponding to the De Rham DGA of differential forms
on some manifold $L$ equipped with a torus action. $\left(\bb_{\cc\left[\vec{\alpha}\right]}^{CW},D+\wedge\right)$
is the Cartan-Weil equivariant extension of this algebra over $\cc\left[\vec{\alpha}\right]$.
$\left(\bb_{\Lambda_{0}^{G}\left(\cc\right)},m\right)$ is some twisted
$A_{\infty}$ algebra which is a deformation of $\left(\bb_{\cc},d+\wedge\right)$,
over the Novikov ring $\Lambda_{0}^{G}\left(\cc\right)$. We assume
that this deformation is $\tb$-invariant; for example, $\left(\bb_{\Lambda_{0}^{G}\left(\cc\right)},m\right)$
might be the quantum deformation associated with the Lagrangian embedding
$\rr\pp^{n}\hookrightarrow\cc\pp^{n}$, in which case the algebra
is $\tb$-invariant because the $\tb$-action on $\rr\pp^{n}$ extends
to $\cc\pp^{n}$ and the associated moduli spaces of discs. Anyway,
if $\left(\bb_{\Lambda_{0}^{G}\left(\cc\right)},m\right)$ is invariant
we can construct $\left(\bb_{\Lambda_{0}^{G}\left(\cc\left[\vec{\alpha}\right]\right)}^{CW},m^{CW}\right)$
which extends both $\left(\bb_{\cc\left[\vec{\alpha}\right]}^{CW},D+\wedge\right)$
and $\left(\bb_{\Lambda_{0}^{G}\left(\cc\right)},m\right)$ over $\Lambda_{0}^{G}\left(\cc\left[\vec{\alpha}\right]\right)$.
See (\ref{eq:square of deformations}) for more details about this
square of twisted $A_{\infty}$ algebras.

When $L$ has even cohomology the equivariant cohomology admits a
perfect pairing (see Corollary \ref{cor:perfect equivariant pairing}),
and we show that the equivariant extension of cyclic twisted $A_{\infty}$
algebras is also cyclic in this case. 

Theorem \ref{thm:equivariant cyclic retraction exists} states that
when $L$ has even cohomology, there exists a cyclic unital retraction
of the Cartan-Weil complex to its cohomology. Such a retraction is
needed in order to construct minimal models for equivariant $A_{\infty}$
algebras. The proof involves the construction of a certain homotopy
kernel (see Definition \ref{def:homotopy kernel} and Proposition
\ref{prop:homotopy kernel exists}). It is a kind of equivariant Hodge-De
Rham decomposition for even cohomology manifolds. This construction
is central to the definition of the equivariant invariants of $\rr\pp^{2n'}\hookrightarrow\cc\pp^{2n'}$
and their computation using fixed point localization, which will be
the subject of a forthcoming paper. 

\medskip{}

\noindent \textbf{Acknowledgments. }I would like to thank my adviser
Jake Solomon. He suggested the problem of open fixed-point localization
and that studying the $A_{\infty}$ formalism in \cite{fukayacyclic}
could be a way to solve it. I also learned a lot from countless conversations
we've had. I would like to thank Mohammed Abouzaid and Rahul Pandharipande
for interesting discussions and ideas. Finally, I'd like to thank
Pavel Giterman for his careful reading of the paper and many useful
suggestions.

\subsection{\label{sub:conventions}Some conventions}

When we say an object $C$ is \emph{graded }we mean it is equipped
with a $\grading$-grading. We call the $\zz$-component of the grading
\emph{the codimension degree }and the\emph{ }$\zz/2$ component of
the grading the \emph{local system degree}. We denote by $\codim x$
(resp. $\ls x$) the codimension degree (resp. the local system degree)
of a homogenous element $x$. We will sometimes use the notation $x^{a,b}$
to indicate that $x^{a,b}\in C^{a,b}$ is homogeneous of degree $\deg x=\left(a,b\right)\in\grading$. 

Maps will be grading-preserving, unless stated otherwise. We will
denote by $C\left[p\right]$ the codimension degree shift of $C$
and by $C\left[p,q\right]$ the bidegree shift of $C$, so 
\[
\left(C\left[p\right]\right)^{a,b}:=C^{a+p,b},\mbox{ and }\left(C\left[p,q\right]\right)^{a,b}:=C^{a+p,b+q}.
\]

Tensor products of graded objects are graded in the usual way. 

We let $\fgvect$ denote the category of \emph{filtered graded real
vector spaces}. The filtration on an object $X$ is denoted $\left\{ \ff^{E}X\right\} $;
it is indexed by $E\in\rr_{\geq0}$ and decreasing: if $E_{1}\leq E_{2}$
we have $\ff^{E_{2}}X\subseteq\ff^{E_{1}}X$. The grading is by ${\zz\oplus\zz/2}$
as above. Morphisms in $\fgvect$ are required to preserve the filtration
and the grading.

An object $V$ is called \emph{discrete }if it is equipped with the
\emph{discrete filtration}, which satisfies $\ff^{>0}V:=\bigcup_{E>0}\ff^{E}V=0$. 

Let $V_{1},V_{2}$ be objects of $\fgvect$. The usual graded tensor
product $V_{1}\otimes V_{2}$ is equipped with a filtration
\begin{equation}
\ff^{E}\left(V_{1}\otimes V_{2}\right):=\bigcup_{E_{1}+E_{2}\geq E}\ff^{E_{1}}V_{1}\otimes\ff^{E_{2}}V_{2}.\label{eq:tensor product filtration}
\end{equation}
This turns $\fgvect$ into a monoidal category. Let $\hfgvect$ denote
the full subcategory of $\fgvect$ of objects which are complete with
respect to the filtration; this category is monoidal with respect
to the completed tensor product $\hotimes$. If $\Lambda$ is a monoid
in $\hfgvect$ we denote by $\hbilam$ the category of\emph{ $\Lambda$-bimodules}
in $\hfgvect$. It is a monoidal category with respect to the monoidal
product $-\hotimes_{\Lambda}-$, where the filtration on $V_{1}\otimes_{\Lambda}V_{2}$
is given by the same formula \ref{eq:tensor product filtration},
except the tensor products on both sides are taken over $\Lambda$. 

If $A$ and $B$ are objects of a category $\cl$ we'll denote by
$\cl\left(A,B\right)$ the internal hom object (assuming it exists).

A \emph{local} \emph{system }on a topological space $X$ is a sheaf
$\lc$ which is locally isomorphic to the \emph{trivial }local system,
which is the constant sheaf $\underline{\cc}$. If $E\overset{\pi}{\longrightarrow}X$
is a rank $r$ vector bundle, the \emph{orientation local system of
$E$}, denoted\emph{ }$Or\left(E\right)$, is the sheafification of
$ $$U\mapsto H_{cv}^{r}\left(E|_{U};\cc\right)$ where $H_{cv}$
denotes the compact vertical cohomology (see Bott and Tu \cite{bott+tu}
pg. 61) it is a local system on $X$. If $X$ is a manifold and $TX$
is the associated tangent bundle, $Or\left(TX\right)$ is called simply
\emph{the orientation local system} \emph{of $X$. An orientation
}for $X$ is an isomorphism $Or\left(TX\right)\simeq\underline{\cc}$,
and given an orientation we say $X$ is \emph{oriented}. If no such
orientation exists we say $X$ is \emph{non-orientable.}

\section{\label{sec:Twisted-A8-algebras}Twisted $A_{\infty}$ algebras}

\subsection{Twisted $A_{\infty}$ algebras as collections of maps}

In this section we define twisted $A_{\infty}$ algebras. We follow
the notation and conventions of \cite{fukayacyclic} and introduce
modifications where they are needed. The differences are summarized
in Remark \ref{rem:fukayacyclic is special case}.
\begin{defn}
\label{def:discrete submonoid}(a) We say a subset $G\subset\rr_{\geq0}\times\zz$
is a \emph{submonoid }if it contains $\mathbf{0}=\left(0,0\right)$
and is closed under addition. If $G$ is a submonoid we denote by
$E:G\to\rr_{\geq0}$ and $\mu:G\to\zz$ the projections to each of
the components. 

(b) We say a submonoid $G\subset\rr_{\geq0}\times\zz$ is a \emph{discrete
submonoid }if the following hold. 
\begin{enumerate}
\item The image $E\left(G\right)\subset\rr_{\geq0}$ is discrete.
\item for each $\lambda\in\rr_{\geq0}$ the inverse image $E^{-1}\left(\lambda\right)$
is a finite set. 
\end{enumerate}
\end{defn}

Let $G$ be a discrete submonoid, $R$ a $2\zz\oplus0$-graded unital
commutative algebra over $\rr$. By that we mean that we think of
$R$ as $\grading$-graded, but $R^{a,b}=0$ unless ${\left(a,b\right)\in2\zz\oplus0\subset\zz\oplus\zz/2\zz}$.
Let $C$ be a $\grading$-graded $R$-module; in particular, the structure
maps are required to respect the grading.
\begin{defn}
\label{def:twisted A8 algebra}A $G$\emph{-gapped twisted} $A_{\infty}$
\emph{algebra structure on $C$ over $R$} is a collection $\left\{ m_{k,\beta}\right\} $
of maps 
\[
m_{k,\beta}:C{}^{\otimes_{R}k}\to C\left[2-k-\mu\left(\beta\right),\mu\left(\beta\right)\mod2\right]
\]
for each $k\in\zz_{\geq0}$ and $\beta\in G$, such that 

\begin{equation}
m_{0,\mathbf{0}}=0,\label{eq:m0,0=00003D0}
\end{equation}
and for every $\beta\in G$ and $k\geq0$ we have 

\begin{multline}
\sum_{k_{1}+k_{2}=k+1}\sum_{\beta_{1}+\beta_{2}=\beta}\sum_{i=1}^{k-k_{2}+1}\left(-1\right)^{*}m_{k_{1},\beta_{1}}\left(x_{1},...,x_{i-1},m_{k_{2},\beta_{2}}\left(x_{i},...,x_{i+k_{2}-1}\right),...,x_{k}\right)=0\label{eq:twisted A8 relations in components}
\end{multline}
for 
\[
*=\sum_{j=1}^{i-1}\left(\mbox{\codim}x_{j}-1\right)+\mu\left(\beta_{2}\right)\sum_{j=1}^{i-1}\left(\ls x_{j}+\codim x_{j}-1\right)+\mu\left(\beta_{2}\right).
\]

\end{defn}
Depending on the context we may simply say that $\left(C,\left\{ m_{k,\beta}\right\} \right)$
is a \emph{twisted $A_{\infty}$ algebra}, or even just an \emph{algebra. }

A pairing $C\otimes_{R}C\overset{\left\langle \cdot\right\rangle }{\longrightarrow}R\left[-p,q\right]$
is called \emph{antisymmetric }if 
\begin{equation}
\left\langle u\otimes v\right\rangle =\left(-1\right)^{1+\left(\codim u-1\right)\left(\codim v-1\right)}\left\langle v\otimes u\right\rangle \label{eq:bracket antisymmetry}
\end{equation}
and \emph{non-degenerate} if for any $u\neq0$ there exists some $v$
such that $\left\langle u,v\right\rangle \neq0$.
\begin{defn}
\label{def:cyclic twisted A8 algebra}We say $\left(C,\left\{ m_{k,\beta}\right\} ,\left\langle \cdot\right\rangle \right)$
is a \emph{cyclic twisted $A_{\infty}$ algebra} if $\left(C,\left\{ m_{k,\beta}\right\} \right)$
is a twisted $A_{\infty}$ algebra and 

(1) $\left\langle \cdot\right\rangle :C\otimes_{R}C\to R\left[-p,q\right]$
is an antisymmetric, non-degenerate pairing,

(2) for every $k\geq0$ and $\beta\in G$ we have 
\begin{equation}
\left\langle m_{k,\beta}\left(x_{1},\cdots,x_{k}\right),x_{0}\right\rangle =\left(-1\right)^{\blacktriangle}\left\langle m_{k,\beta}\left(x_{0},\cdots,x_{k-1}\right),x_{k}\right\rangle \label{eq:twisted cyclic symmetry}
\end{equation}
for $\blacktriangle=\left(\codim x_{0}-1\right)\left(\sum_{j=1}^{k}\left(\codim x_{j}-1\right)\right)+\mu\left(\beta\right)\ls x_{0}$,
and 

(3) the induced pairing on $HC=H\left(C,m_{1,\mathbf{0}}\right)$
is perfect; that is the induced map $HC\to\mathbf{Mod}_{R}\left(HC,R\left[-p,q\right]\right)$
is an isomorphism. Here $\mathbf{Mod}_{R}$ is the category of graded
$R$-modules (cf. $\S$ \ref{sub:conventions}).
\end{defn}
Note that $m_{1,\mathbf{0}}^{2}=0$, so $HC$ is well-defined, and
$\left\langle m_{1,\mathbf{0}}x,y\right\rangle \pm\left\langle m_{1,\mathbf{0}}y,x\right\rangle =0$,
so there is indeed an induced pairing on $HC$.
\begin{rem}
\label{rem:fukayacyclic is special case}Our definition \ref{def:discrete submonoid}
of a discrete submonoid differs from Definition 6.2 in \cite{fukayacyclic}
in that $\mu$ can obtain odd values. If $\mu\left(G\right)\subset2\zz$,
it is immediately apparent that the sign in Eq (\ref{eq:twisted A8 relations in components})
reduces to the sign in the filtered $A_{\infty}$ relation (61) in
\cite{fukayacyclic}. Similarly the sign in Eq (\ref{eq:twisted cyclic symmetry})
reduces to the sign in the cyclic symmetry condition (62) in \cite{fukayacyclic}.
In fact, for $\mu\left(G\right)\subset2\zz$ we find that $G$-gapped
cyclic twisted $A_{\infty}$ algebra over $R=\rr$ are in bijection
with $G$-gapped cyclic filtered $A_{\infty}$ algebras in the sense
of Definition 6.4 of \cite{fukayacyclic}.
\end{rem}

\begin{defn}
An element $\mathbf{e}\in C^{0,0}$ is a \emph{strict unit }of the
twisted $A_{\infty}$ algebra $\left(C,\left\{ m_{k,\beta}\right\} \right)$
if 
\begin{equation}
m_{2,\mathbf{0}}\left(\mathbf{e},x\right)=\left(-1\right)^{\codim x}m_{2,\mathbf{0}}\left(x,\mathbf{e}\right)=x\label{eq:unit relation}
\end{equation}
 and $m_{k,\beta}\left(\cdots,\mathbf{e},\cdots\right)=0$ for all
$\left(k,\beta\right)\neq\left(2,\mathbf{0}\right)$.
\end{defn}

\begin{example}
\label{exa:DGA}Suppose $C$ is a $\grading$-graded unital associative
algebra over $R$ with product $\wedge$, and $d:C\to C\left[1\right]=C\left[1,0\right]$
(cf. $\S$ \ref{sub:conventions}) is an $R$-linear differential
satisfying the graded Leibniz rule 
\[
d\left(x\wedge y\right)=dx\wedge y+\left(-1\right)^{\codim x}x\wedge dy.
\]
We will say that $\left(C,d,\wedge\right)$ is an $R$\emph{-DGA},
or a \emph{DGA over $R$},\emph{ }in this case\emph{.} In this case
we can take the trivial discrete submonoid $G=\left\{ \mathbf{0}\right\} $
and construct a unital twisted $A_{\infty}$ algebra $\left(C,\left\{ m_{k,\beta}\right\} \right)$
over $G$ by setting 
\begin{equation}
m_{1,\mathbf{0}}\left(x\right)=\left(-1\right)^{\codim x}dx\label{eq:m1,0 vs differential}
\end{equation}
\begin{equation}
m_{2,\mathbf{0}}\left(x,y\right)=\left(-1\right)^{\codim x+\codim x\cdot\codim y}x\wedge y\label{eq:m2,0 vs. wedge}
\end{equation}
and $m_{k,\mathbf{0}}=0$ for $k\not\neq1,2$.

If $\int:C\to R\left[-p,q\mod2\right]$ is any $R$-linear map (for
some $p\in\zz$ and $q\in\left\{ 0,1\right\} $) and we set 
\begin{equation}
\left\langle x\otimes y\right\rangle :=\left(-1\right)^{\codim x\cdot\codim y+\codim x}\int x\wedge y,\label{eq:pairing from functional}
\end{equation}
then $\left(C,\left\{ m_{k,\beta}\right\} \right)$ satisfies Eq (\ref{eq:twisted cyclic symmetry}).
It follows that $\left(C,\left\{ m_{k,\beta}\right\} ,\left\langle \cdot\right\rangle \right)$
is cyclic if $\left\langle \cdot\right\rangle $ is non-degenerate
and the induced pairing on $H\left(C,d\right)$ is perfect (making
$H\left(C,d\right)$ a Frobenius algebra). In this case we say $\left(C,d,\wedge,\int\right)$
is a \emph{cyclic DGA} \emph{(over $R$)}.
\end{example}

\begin{example}
\label{exa:De Rham DGA}Here is a special case of Example \ref{exa:DGA}
that will be important. Set $R=\cc$. Let $L$ be a closed, non-orientable
manifold. Take $C\left(L\right)$ to be the $\grading$-graded $\cc$
vector space with 
\[
C^{a,b}\left(L\right)=\Omega^{a}\left(L;Or\left(TL\right)^{\otimes_{\cc}b}\right),
\]
the smooth differential $a$-forms forms on $L$ with values in the
local system $Or\left(TL\right)$ (cf. $\S$ \ref{sub:conventions}).
The exterior derivative $d$ and wedge product $\wedge$ make $C\left(L\right)$
a DGA over $R=\cc$. Integration $\int:C\to\cc\left[-n,1\mod2\right]$
turns it into a cyclic DGA, where we set $\int\omega=0$ unless $\omega\in C^{n,1}\left(L\right)$. \end{example}
\begin{defn}
\label{def:deformation}If $\left(C,d,\wedge\right)$ is an $R$-DGA,
a $G$-gapped twisted $A_{\infty}$ algebra structure $\left\{ m_{k,\beta}\right\} $
on $C$ over $R$ will be called a\emph{ deformation of }$\left(C,d,\wedge\right)$
if $m_{1,\mathbf{0}}$ and $m_{2,\mathbf{0}}$ are given by Eqs (\ref{eq:m1,0 vs differential},
\ref{eq:m2,0 vs. wedge}) and $m_{k,\mathbf{0}}=0$ for $k\geq3$.
Similarly, a \emph{cyclic deformation }of a cyclic DGA $\left(C,d,\wedge,\int\right)$
is a cyclic twisted $A_{\infty}$ algebra $\left(C,\left\{ m_{k,\beta}\right\} ,\left\langle \cdot\right\rangle \right)$
where $\left\langle \cdot\right\rangle $ is defined by Eq (\ref{eq:pairing from functional}),
such that $\left(C,\left\{ m_{k,\beta}\right\} \right)$ is a deformation
of $\left(C,d,\wedge\right)$. A (cyclic) deformation is called \emph{unital}
if the unit of the DGA is also a unit for the twisted $A_{\infty}$
algebra.
\end{defn}
For a justification for this terminology, see Remark \ref{rem:deformation explained}.

\subsection{\label{sub:bar coalgebra}Twisted $A_{\infty}$ algebras as differentials;
morphisms and homotopies}

To better see what's going on, and give cleaner and more general (or
at least easy to generalize) definitions and proofs, we reinterpret
twisted $A_{\infty}$ algebras as differentials on a certain bar coalgebra,
see Proposition \ref{prop:Differentials are twisted A8 algebras}
below (this is the \emph{easy-to-use }equivalent definition we referred
to in the introduction).

Fix a discrete submonoid $G$ and a $(2\zz\oplus0)$-graded commutative
unital real algebra $R$. Consider the \emph{Novikov ring} of formal,
possibly infinite, sums  
\[
\Lambda=\Lambda_{0}^{G}\left(R\right):=\left\{ \sum_{\beta\in G}a_{\beta}\epsilon^{\mu\left(\beta\right)}T^{E\left(\beta\right)}|a_{\beta}\in R\right\} 
\]
(as a set, this is just the set of maps $R^{G}$. The product is defined
using the addition in $G$). It is complete with respect to the filtration
of ideals $\left\{ \ff^{E}\Lambda\right\} _{E\in\rr_{\geq0}}$ given
by 
\[
\ff^{E}\Lambda:=\left\{ \sum_{\left\{ \beta|E\left(\beta\right)\geq E\right\} }a_{\beta}\epsilon^{\mu\left(\beta\right)}T^{E\left(\beta\right)}|a_{\beta}\in R\right\} 
\]
We give $\Lambda$ a $\zz\oplus0\subset\grading$ grading by setting
$\deg\epsilon=\left(1,0\right)$ and $\deg T=\left(0,0\right)$. We
denote by $1_{\Lambda}\in\Lambda$ the unit of $\Lambda$. Thus, it
becomes a monoid in $\hfgvect$, and we denote by $\hbilam$ the category
of $\Lambda$-bimodules - see $\S$\ref{sub:conventions} for precise
explanation of what this means. 

Now let $C$ be a $\grading$-graded $R$-module, taken with the discrete
filtration (see $\S$\ref{sub:conventions}). We construct an object
$C^{G}\in\hbilam$ by setting $C^{G}:=\Lambda_{0}^{G}\left(R\right)\hotimes_{R}C$
with the usual grading and filtration on the tensor product. The bimodule
structure is defined in the following, non-symmetric, way:
\[
\epsilon^{r_{1}}T^{E_{1}}\otimes\left(c\epsilon^{r}T^{E}\right)\otimes\epsilon^{r_{2}}T^{E_{2}}\mapsto\left(-1\right)^{\left(\codim c+\ls c+1\right)r_{1}}c\epsilon^{r_{1}+r+r_{2}}T^{E_{1}+E+E_{2}}
\]

\begin{rem}
If $R=\rr$ and $\mu\left(G\right)\subset2\zz$ we can identify $\Lambda_{0}^{G}\left(R\right)$
with $\Lambda_{0}^{G}$ in Definition 6.3 of \cite{fukayacyclic}
by setting $e=\epsilon^{2}$. In this case $C^{G}$ becomes a symmetric
bimodule (i.e. a $\Lambda$-module). See also Remark \ref{rem:fukayacyclic is special case}.
\end{rem}
We now define the \emph{bar coalgebra} $\bb\left(C^{G}\right)$\emph{
associated with the bimodule} $C^{G}$. As an object of $\hbilam$
we have

\[
\bb\left(C^{G}\right):=\widehat{\bigoplus}_{k\geq0}\left(C^{G}\left[1\right]\right)^{\hotimes_{\Lambda}k}
\]
where $\widehat{\oplus}$ is the coproduct in $\hbilam$, so $ $
\[
\bb\left(C^{G}\right)=\left\{ \left(x_{0},...,x_{k},...\right)|\exists E_{i}\to\infty\mbox{ s.t. }x_{i}\in\ff^{E_{i}}\left(C^{G}\left[1\right]\right)^{\hotimes_{\Lambda}k}\right\} 
\]
(cf. Definition 3.2.16 in \cite{FOOO}). We denote by $i_{j}:\left(C^{G}\left[1\right]\right)^{\hotimes_{\Lambda}j}\to\bb\left(C^{G}\left[1\right]\right)$
and $\pi_{j}:\bb\left(C^{G}\left[1\right]\right)\to\left(C^{G}\left[1\right]\right)^{\hotimes_{\Lambda}j}$
the structure maps associated with the coproduct $\widehat{\bigoplus}$.
The \emph{comultiplication} $\Delta:\bb\left(C^{G}\right)\to\bb\left(C^{G}\right)\hotimes_{\Lambda}\bb\left(C^{G}\right)$
is defined by $\Delta\left(x_{1}\otimes\cdots\otimes x_{k}\right)=\sum_{i=0}^{k}\left(x_{1}\otimes\cdots\otimes x_{i}\right)\overline{\otimes}\left(x_{i+1}\otimes\cdots\otimes x_{k}\right)$,
and the \emph{counit }$\eta$ is the projection $\pi_{0}:\bb\left(C^{G}\right)\to\Lambda$.
$\left(\bb\left(C^{G}\right),\Delta,\eta\right)$ is a \emph{coalgebra
}in the sense that%
\footnote{Indeed it is not hard to show $\im\Delta$ is \emph{not }contained
in the incomplete tensor product $\bb\left(C^{G}\right)\otimes_{\Lambda}\bb\left(C^{G}\right)$,
generally speaking. So this is not the same as a coalgebra in the
usual sense which ``happens to be'' complete. With this caveat pointed
out, we will none-the-less find it convenient to refer to it simply
as a coalgebra.%
} it is a comonoid in the category $\hbilam$.

Morphisms between coalgebras are defined in the usual way, and are
always assumed to be counital (dualizing the notion of unital algebra
morphism).
\begin{defn}
If $f_{1},f_{2}:\left(Q,\Delta,\eta\right)\to\left(Q',\Delta',\eta'\right)$
are two coalgebra morphisms, a map $h:Q\left[-d\right]\to Q'$ will
be called an\emph{ $\left(f_{1},f_{2}\right)$-coderivation of degree
$d$} if

\[
\eta'h=0
\]
(so for us, coderivations are always counital) and
\[
\Delta\circ h=\left[\left(h\otimes f_{2}\right)\pm\left(f_{1}\otimes h\right)\right]\circ\Delta
\]
where 
\begin{equation}
\left(\left(h\otimes f_{2}\right)\pm\left(f_{1}\otimes h\right)\right)\left(x\otimes y\right):=hx\otimes f_{2}y+\left(-1\right)^{d\cdot\codim x}f_{1}x\otimes hy.\label{eq:graded coderivation sign}
\end{equation}

A \emph{differential} on a coalgebra $\left(Q,\Delta,\eta\right)$
is an $\left(\mbox{id}_{Q},\mbox{id}_{Q}\right)$-coderivation ${m:Q\left[-1\right]\to Q}$
of degree 1 such that $m^{2}=0$.
\end{defn}

Let $C$ be an $R$-module and let $\bb=\bb\left(C^{G}\right)$ be
the corresponding bar coalgebra. Let $\mbox{Coder}^{1}\left(\bb\right)$
denote the set of $\left(\mbox{id}_{\bb},\mbox{id}_{\bb}\right)$-coderivations
of degree 1. We have the following bijections of sets:

\begin{multline}
\mbox{Coder}^{1}\left(\bb\left(C^{G}\right)\right)\overset{\left(1\right)}{\simeq}\hbilam\left(\bb,C^{G}\left[2\right]\right)\overset{\left(2\right)}{\simeq}\prod_{k\geq0}\hbilam\left(\left(C{}^{G}\left[1\right]\right)^{\hotimes_{\Lambda}k},C^{G}\left[2\right]\right)\\
\overset{\left(3\right)}{\simeq}\prod_{k\geq0}\mathbf{Mod}_{R}\left(C^{\otimes_{R}k},C^{G}\left[2-k\right]\right)\overset{\left(4\right)}{\simeq}\prod_{k\geq0,\beta\in G}\mathbf{Mod}_{R}\left(C^{\otimes_{R}k},C\left[2-k-\mu\left(\beta\right)\right]\right).\label{eq:coderivations in components}
\end{multline}
The bijection (1) is given by the maps $m\mapsto\pi_{1}\circ m$ and
${\Xi^{1}\circ\left(\mbox{id}_{\bb}\otimes\rho\otimes\mbox{id}_{\bb}\right)\circ\Delta^{\circ2}\mapsfrom\rho}$
where $\Delta^{\circ2}:\bb\to\bb\hotimes_{\Lambda}\bb\hotimes_{\Lambda}\bb$
is the reiteration of $\Delta$ and ${\Xi^{d}:\bb\hotimes_{\Lambda}C^{G}\left[1+d\right]\hotimes_{\Lambda}\bb\to\bb\left[d\right]}$
is the unique continuous additive map which satisfies 
\[
\Xi^{d}\left(\left(x_{1}\otimes\cdots\otimes x_{a}\right)\overline{\otimes}y\overline{\otimes}\left(z_{1}\otimes\cdots\otimes z_{b}\right)\right)=\left(-1\right)^{d\cdot\sum_{i=1}^{a}\left(\codim x_{i}-1\right)}\cdots\otimes x_{a}\otimes y\otimes z_{1}\otimes\cdots.
\]
(2) comes from the direct sum decomposition $\bb=\widehat{\oplus}_{k\geq0}\left(C{}^{G}\left[1\right]\right)^{\hotimes_{\Lambda}k}$.
(3) is the extension/restriction of scalars adjunction for $R\to\Lambda$,
with a $k$-degree shift, and (4) comes from the isomorphism of $R$-modules
$C^{G}=\prod_{\beta\in G}\epsilon^{\mu\left(\beta\right)}\tau^{E\left(\beta\right)}C$.

\begin{defn}
A coderivation $m:\bb\left(C^{G}\right)\to\bb\left(C^{G}\right)$
will be called \emph{tame }if $\left(\pi_{1}\circ m\circ i_{0}\right)\left(1_{\Lambda}\right)\in\ff^{>0}\left(C^{G}\left[1\right]\right)$.

If $C'$ is another $R$-module with $\bb'=\bb\left(C'^{G}\right)$
the corresponding bar coalgebra, a morphism $f:\bb\to\bb'$ will be
called \emph{tame }if $\left(\pi_{1}'\circ f\circ i_{0}\right)\left(1_{\Lambda}\right)\in\ff^{>0}\left(C'^{G}\left[1\right]\right)$,
where $\pi_{1}':\bb'\to C'^{G}\left[1\right]$ is the projection.
\end{defn}
The following simple proposition is important, in that it allows us
to redefine twisted $A_{\infty}$ algebras as tame differentials on
the bar coalgebra. We will work from this vantage point in the remainder
of the paper.
\begin{prop}
\label{prop:Differentials are twisted A8 algebras}Let $G$ be a discrete
submonoid, $R$ a commutative unital real $\left(2\zz\oplus0\right)$-graded
algebra, and $C$ a $\grading$-graded $R$-module. Under the bijection
(\ref{eq:coderivations in components}), tame differentials $m$ on
$\bb=\bb\left(C^{G}\right)$ correspond to twisted $A_{\infty}$ algebra
structures $\left\{ m_{k,\beta}\right\} $ on $C$.\end{prop}
\begin{proof}
Let $m\in\mbox{Coder}^{1}\left(\bb\right)$. For $x_{1},...,x_{k}\in C$
we write 
\[
\left(\pi_{1}\circ m\circ m\right)\left(x_{1}\otimes\cdots\otimes x_{k}\right)=\sum_{\beta\in G}\lambda\left(\beta\right)\mathsf{C}\left(x_{1},...,x_{k};\beta\right)
\]
 for $\mathsf{C}\left(x_{1},...,x_{k};\beta\right)\in C$ and $\lambda\left(\beta\right):=\epsilon^{\mu\left(\beta\right)}T^{E\left(\beta\right)}\in\Lambda^{\mu\left(\beta\right),0}$.
Clearly $m^{2}=0$ iff $\mathsf{C}\left(x_{1},...,x_{k};\beta\right)=0$
for all $x_{1},...,x_{k}\in C$ and $\beta\in G$. We compute $\mathsf{C}\left(x_{1},...,x_{k};\beta\right)$
as follows. 

\begin{multline*}
\lambda\left(\beta\right)\mathsf{C}\left(x_{1},...,x_{k};\beta\right)=\\
=\sum\left(-1\right)^{\wasylozenge_{1}}\lambda\left(\beta_{1}\right)m_{k_{1},\beta_{1}}\big(x_{1}\otimes\cdots\otimes x_{i-1}\otimes\\
\otimes\lambda\left(\beta_{2}\right)m_{k_{2},\beta_{2}}\left(x_{i}\otimes\cdots\otimes x_{i+k_{2}-1}\right)\otimes x_{i+k_{2}}\otimes\cdots\otimes x_{k}\big)\\
=\lambda\left(\beta_{1}\right)\sum\left(-1\right)^{\wasylozenge_{1}+\wasylozenge_{2}}m_{k_{1},\beta_{1}}\left(\lambda\left(\beta_{2}\right)x_{1}\otimes\cdots\otimes m_{k_{2},\beta_{2}}\left(x_{i}\otimes\cdots\otimes x_{i+k_{2}-1}\right)\otimes\cdots\right)=\\
=\lambda\left(\beta\right)\sum\left(-1\right)^{\wasylozenge_{1}+\wasylozenge_{2}+\wasylozenge_{3}}m_{k_{1},\beta_{1}}\left(x_{1}\otimes\cdots\otimes m_{k_{2},\beta_{2}}\left(x_{i}\otimes\cdots\otimes x_{i+k_{2}-1}\right)\otimes\cdots\right)
\end{multline*}
where all the sums range over 
\[
\left\{ \left(k_{1},k_{2},\beta_{1},\beta_{2},i\right)|k_{1}+k_{2}=k+1,\;\beta_{1}+\beta_{2}=\beta,\:1\leq i\leq k-k_{2}+1\right\} ,
\]
and the signs are 

\begin{align*}
\wasylozenge_{1} & =\sum_{j=1}^{i-1}\left(\codim x_{j}-1\right)\\
\wasylozenge_{2} & =\sum_{i=1}^{j-1}\left(\codim x_{j}+\ls x_{j}+1\right)\mu\left(\beta_{2}\right)\\
\wasylozenge_{3} & =\mu\left(\beta_{2}\right)-1
\end{align*}
Clearly, the vanishing of $\mathsf{C}\left(x_{1},...,x_{k};\beta\right)$
is equivalent to Eq (\ref{eq:twisted A8 relations in components}).
 Requiring $m$ to be tame is tantamount to Eq (\ref{eq:m0,0=00003D0}).
\end{proof}
We will often refer to a pair $\left(\bb,m\right)$ as a twisted $A_{\infty}$
algebra; by that we mean that $\bb=\bb\left(C^{G}\right)$ for some
$C$ and $m$ is a tame differential on $\bb$. 
\begin{rem}
\label{rem:deformation explained}Let $G,G'$ be two discrete submonoids,
and let $R,R'$ be two unital commutative $\left(2\zz\oplus0\right)$-graded
real algebras. Denote $\Lambda:=\Lambda_{0}^{G}\left(R\right)$ and
$\Lambda'=\Lambda_{0}^{G'}\left(R'\right)$, and suppose we have a
map $\Lambda\to\Lambda'$. If $C$ is any $R$-module, we have a natural
isomorphism 
\[
\bb_{\Lambda}\left(C\hotimes_{R}\Lambda\right)\otimes_{\Lambda}\Lambda'=\bb_{\Lambda'}\left(C\hotimes_{R}\Lambda'\right)
\]
which commute with $\Delta,\eta$ in the obvious way. Hereafter, we
use a subscript as in $\bb_{\Lambda}\left(C\hotimes_{R}\Lambda\right)$
to denote the underlying $\hfgvect$ monoid over which the bar complex
is constructed, unless $\Lambda$ is clear from the context. It follows
that any differential $m$ on $\bb_{\Lambda}\left(C\hotimes_{R}\Lambda\right)$
defines a differential $m':=m\otimes\mbox{id}_{\Lambda'}$ on $\bb_{\Lambda'}\left(C\hotimes_{R}\Lambda'\right)$.

A special case of the above is when we take $G'=\left\{ 0\right\} $,
$R'=R$, so $\Lambda'=R$. We then have a quotient map $\Lambda\to\Lambda'=R$
obtained by setting $T=0,\epsilon=0$. If $\left(C,\left\{ m_{k,\beta}\right\} \right)$
is a $G$-gapped twisted $A_{\infty}$ algebra over $R$, which is
a deformation of some DGA $\left(C,d,\wedge\right)$ (cf. Definition
\ref{def:deformation}) then $m'=m\otimes\mbox{id}_{R}$ is the twisted
$A_{\infty}$ algebra which corresponds to $\left(C,d,\wedge\right)$
as in Example \ref{exa:DGA}.\end{rem}
\begin{defn}
(a) Let $\left(\bb,m\right),\left(\bb',m'\right)$ be twisted $A_{\infty}$
algebras. A \emph{twisted $A_{\infty}$ morphism }$f:\left(\bb,m\right)\to\left(\bb',m'\right)$
is a tame coalgebra morphism $f:\bb\to\bb'$ which is a chain map:
$f\circ m=m'\circ f$. 

(b) Given two twisted $A_{\infty}$ morphisms $f_{1},f_{2}:\left(\bb,m\right)\to\left(\bb',m'\right)$
a \emph{homotopy $h:f_{1}\Rightarrow f_{2}$ }is an $\left(f_{1},f_{2}\right)$-coderivation
$h:\bb\left[+1\right]\to\bb'$ of degree $\left(-1\right)$ with $m'h+hm=f_{2}-f_{1}$. \end{defn}
\begin{rem}
\label{rem:general coderivation in components}The bijection (\ref{eq:coderivations in components})
generalizes easily to $\left(f_{1},f_{2}\right)$-coderivations of
any degree $d$ if we set $\Xi^{d}\circ\left(f_{1}\otimes\rho\otimes f_{2}\right)\circ\Delta^{2}\mapsfrom\rho$. 

If $\bb=\bb\left(C^{G}\right)$ and $\bb'=\bb\left(C'^{G}\right)$
are two bar coalgebras and we let $Mor\left(\bb,\bb'\right)$ denote
the set of \emph{tame }coalgebra morphisms $f:\bb\to\bb'$, then we
have a bijection $f\mapsto\left\{ f_{k,\beta}|f_{0,\mathbf{0}}=0\right\} $
\begin{equation}
Mor\left(\bb,\bb'\right)\simeq\prod_{\left(k,\beta\right)\neq\left(0,\mathbf{0}\right)}\mathbf{Mod}_{R}\left(C^{\otimes_{R}k},C'\left[1-k-\mu\left(\beta\right)\right]\right)\label{eq:morphism in components}
\end{equation}
The maps $f_{k,\beta}:C^{\otimes_{R}k}\to C'\left[1-k-\mu\left(\beta\right)\right]$
are uniquely determined by the following equation:

\[
\pi_{1}'\circ f\circ i_{k}=\sum_{\beta}\epsilon^{\mu\left(\beta\right)}T^{E\left(\beta\right)}f_{k,\beta}
\]
where $i_{k}:\left(C^{G}\right)^{\hotimes_{\Lambda}k}\to\bb\left(C^{G}\right)$
and $\pi_{1}':\bb'\to C'^{G}$ are the coproduct structure maps.

This is proved by writing down bijections, similar to Eq (\ref{eq:coderivations in components}).
The bijection (1) in Eq (\ref{eq:coderivations in components}) is
replaced by the bijection of sets 
\[
Mor\left(\bb\left(C^{G}\right),\bb\left(C'^{G}\right)\right)\simeq\left\{ \rho:\bb\left(C^{G}\right)\to C'^{G}\left[1\right]:\left(\rho\circ i_{1}\right)\left(1_{\Lambda}\right)\in\ff^{>0}\left(C'^{G}\right)\right\} 
\]
where $\rho$ maps to the morphism

\begin{equation}
\hat{\rho}:=\sum_{m\geq0}i_{m+1}'\circ\rho^{\otimes\left(m+1\right)}\circ\Delta^{m}\label{eq:reconstructing a morphism}
\end{equation}
where $i_{m+1}':\left(C'^{G}\left[1\right]\right)^{\hotimes_{\Lambda}\left(m+1\right)}\to\bb'\left(C^{G}\left[1\right]\right)$
is the coproduct structure map. We need to require $\left(\rho\circ i_{1}\right)\left(1_{\Lambda}\right)\in\ff^{>0}\left(C'^{G}\right)$
for the sum in (\ref{eq:reconstructing a morphism}) to converge;
see also the discussion following Eq (3.2.28) in \cite{FOOO} (the
requirement that \emph{differentials} be tame will be used only later,
in the proof of the homological perturbation lemma, Theorem \ref{thm:homological perturbation lemma}).
The other modifications to Eq (\ref{eq:coderivations in components})
are straightforward.

At any rate, using these bijections one can spell out everything components:
the twisted $A_{\infty}$ morphism relation $f\circ m=m'\circ f$,
the formula for the composition of two morphisms $f_{1}\circ f_{2}$,
and the homotopy equation $m'h+hm=f_{2}-f_{1}$, to name a few relations.
Since we will avoid working in components, we only illustrate this
principle with the following proposition.\end{rem}
\begin{prop}
\label{prop:morphism in components}Let $\left(C,\left\{ m_{k,\beta}\right\} \right)$
and $\left(C',\left\{ m_{k,\beta}'\right\} \right)$ be twisted $A_{\infty}$
algebras, with corresponding differential coalgebras $\left(\bb,m\right)$
and $\left(\bb',m'\right)$. The bijection \ref{eq:morphism in components}
induces a bijection from the set of twisted $A_{\infty}$ morphisms
$f:\left(\bb,m\right)\to\left(\bb',m'\right)$ and sets of $R$-module
maps 
\[
\left\{ f_{k,\beta}:C^{\otimes_{R}k}\to C'\left[1-k-\mu\left(\beta\right)\right]|f_{0,\mathbf{0}}=0\right\} _{k\geq0,\beta\in G}
\]
such that for any $k\geq0,\beta\in G$ and $x_{1},...,x_{k}\in C$
we have

\begin{multline}
\sum\left(-1\right)^{\wasylozenge_{L}}m_{r,\beta_{0}}\left(f_{k_{1},\beta_{1}}\left(x_{1},...,x_{k_{1}}\right),...,f_{k_{r},\beta_{r}}\left(x_{k-k_{r}+1},...,x_{k}\right)\right)=\\
=\sum\left(-1\right)^{\wasylozenge_{R}}f_{k_{1},\beta_{1}}\left(x_{1},...,m_{k_{2},\beta_{2}}\left(x_{i},...,x_{i+k_{2}-1}\right),...,x_{k}\right).\label{eq:morphism relation}
\end{multline}
The sum on the left hand side ranges over all $r$-tuples of pairs
$\left(\left(k_{1},\beta_{1}\right),...,\left(k_{r},\beta_{r}\right)\right)$
and $\beta_{0}\in G$ with $\sum_{j=1}^{r}k_{j}=k$ and $\beta_{0}+\sum_{j=1}^{r}\beta_{j}=\beta$.
The sum on the right hand side ranges over all $k_{1},k_{2}\geq0$
and $\beta_{1},\beta_{2}\in G$ such that $k_{1}+k_{2}=k+1$ and $\beta_{1}+\beta_{2}=\beta$.
The signs are as follows.

\[
\wasylozenge_{L}=\sum_{i=1}^{r}\mu\left(\beta_{i}\right)+\sum_{b=1}^{r}\mu\left(\beta_{b}\right)\cdot\left(\left(b-1\right)+\sum_{j\leq\sum_{i=1}^{b-1}k_{i}}\left(\ls x_{j}+\codim x_{j}-1\right)\right)
\]
\[
\wasylozenge_{R}=\sum_{j=1}^{i-1}\left(\codim x_{j}-1\right)+\mu\left(\beta_{2}\right)\cdot\sum_{j=1}^{i-1}\left(1+\ls x_{j}+\codim x_{j}\right).
\]

\end{prop}

\begin{proof}
Straightforward.
\end{proof}
We will find it convenient to denote twisted $A_{\infty}$ morphisms
also as ${f:\left(C,\left\{ m_{k,\beta}\right\} \right)\to\left(C',\left\{ m_{k,\beta}'\right\} \right)}$.
This always means a map of the corresponding bar coalgebras or, equivalently,
a set of $R$-module maps $\left\{ f_{k,\beta}\right\} $ as in ${\mbox{Proposition }\ref{prop:morphism in components}}$.

\subsection{Pairing cocycles, cyclic and unital morphisms}

Let $G$ be a discrete submonoid, $R$ a commutative unital $\left(2\zz\oplus0\right)$-graded
real algebra. Let $\left(\bb,m\right)$ be a $G$-gapped twisted $A_{\infty}$
algebra over $R$. Recall this means $\bb=\bb\left(C^{G}\right)$
for some $\grading$-graded $R$-module $C$. 

We define $CC\left(\bb\right):=\hbilam\left(\bb,\Lambda\right)$ to
be the complex of $\Lambda$-bimodule maps equipped with the differential
$m^{*}:CC\left(\bb\right)\to CC\left(\bb\right)\left[1\right]$ defined
by $m^{*}\phi=\phi\circ m$. Note every pairing $C\otimes_{R}C\overset{\left\langle \cdot\right\rangle }{\longrightarrow}R\left[-p,q\mod2\right]$
defines an element $\lozenge=\left\langle \cdot\right\rangle \otimes\mbox{id}_{\Lambda}\in CC^{-p,q}\left(\bb\right)$
by setting 
\[
\lozenge\left(x_{1}\otimes x_{2}\right)=\left\langle x_{1},x_{2}\right\rangle 
\]
for $x_{i}\in C\subset C^{G}$ and 
\[
\lozenge\left(\mathbf{x}\right)=0
\]
for $\mathbf{x}\in\left(C^{G}\left[1\right]\right)^{\hotimes_{\Lambda}k}$
for $k\neq2$. Given $\lozenge\in CC^{-p,q}\left(\bb\right)$ we can
``read off'' the unique pairing $\left\langle \cdot\right\rangle $
which produced it. 
\begin{defn}
An element $\lozenge\in CC^{-p,q}\left(\bb\right)$ will be called
a \emph{pairing cocycle for $\left(\bb,m\right)$} \emph{of degree
$\left(p,q\right)$ }if $\lozenge=\left\langle \cdot\right\rangle \otimes\mbox{id}_{\Lambda}$
where ${\left\langle \cdot\right\rangle :C\otimes_{R}C\to R\left[-p,q\mod2\right]}$
is an antisymmetric, non-degenerate pairing which induces a perfect
pairing on $H\left(C,m_{1,\mathbf{0}}\right)$, and we have 
\[
m^{*}\lozenge=0.
\]

\end{defn}
We then have
\begin{prop}
Pairing cocycles are in bijection with pairings ${\left\langle \cdot\right\rangle :C\otimes_{R}C\to R\left[-p,q\mod2\right]}$
such that $\left(C,\left\{ m_{k,\beta}\right\} ,\left\langle \cdot\right\rangle \right)$
is a cyclic twisted $A_{\infty}$ algebra (cf. Definition \ref{def:cyclic twisted A8 algebra}). \end{prop}
\begin{proof}
A straightforward computation shows equation (\ref{eq:twisted cyclic symmetry})
is equivalent to ${m^{*}\lozenge=0}$. 
\end{proof}
Thus we may refer to $\left(\bb,m,\lozenge\right)$ as a cyclic twisted
$A_{\infty}$ algebra. Let $\left(\bb,m,\lozenge\right),\left(\bb',m',\lozenge'\right)$
be cyclic twisted $A_{\infty}$ algebras.
\begin{defn}
\label{def:cyclic and unital morphisms}(a) A \emph{cyclic morphism}
$f:\left(\bb,m,\lozenge\right)\to\left(\bb',m',\lozenge'\right)$
is a twisted $A_{\infty}$ morphism $\left(\bb,m\right)\to\left(\bb',m'\right)$
such that $f^{*}\lozenge'=\lozenge$. 

(b) If $\mathbf{e}\in C^{0,0}\subset C^{G}\subset\bb$ (respectively,
$\mathbf{e}'\in C'^{0,0}\subset\bb'$) is a unit for $\left(\bb,m\right)$
(resp. $\left(\bb',m'\right)$), a morphism $f:\left(\bb,m\right)\to\left(\bb',m'\right)$
will be called \emph{unital} if $f\mathbf{e}=\mathbf{e}'$.\end{defn}
\begin{rem}
If $\mu\left(G\right)\subset2\zz$ it is not hard to see that a morphism
$\left\{ f_{k,\beta}\right\} $ is cyclic (resp. unital) iff Eqs (73,74)
(resp. (72)) in \cite{fukayacyclic} hold.
\end{rem}

\begin{rem}
For the purposes of this paper, pairing cocycles will serve just as
a convenient book-keeping device. In future work we would like to
use them to give a more meaningful definition of the homotopy theory
of cyclic twisted $A_{\infty}$ algebras along the lines of Kontsevich
and Soibelman's work \cite{k+sA8}.
\end{rem}

\section{\label{sec:Homological-algebra}Homological algebra of twisted $A_{\infty}$
algebras}

In this section we will show that under certain assumptions, one
can construct minimal models for (possibly cyclic or unital) twisted
$A_{\infty}$ algebras. This is a central tool for analyzing the homotopy
theory of twisted $A_{\infty}$ algebras. 

\medskip{}

Let $\left(C,d\right)$ be a \emph{dg $R$-module}. Namely, $C$ is
a $\grading$-graded $R$-module and $d:C\to C\left[1\right]$ is
an $R$-linear map which squares to zero. We denote $HC=H\left(C,d\right)$
and define $d'x:=\left(-1\right)^{\codim x}dx$.

\begin{defn}
A twisted $A_{\infty}$ algebra $\left(C,\left\{ m_{k,\beta}\right\} \right)$
is called a \emph{perturbation }of $\left(C,d\right)$ if $m_{1,\mathbf{0}}=d'$. 
\end{defn}
We will construct a minimal model for a perturbation of $\left(C,d\right)$
by transferring the coalgebra differential to $HC:=H\left(C,d\right)$.
In order to carry this out we need some auxiliary data.
\begin{defn}
\label{def:retraction}A \emph{retraction of $\left(C,d\right)$ to
$HC$ }is a 3-tuple $\left(\Pi_{1,\mathbf{0}},I_{1,\mathbf{0}},h_{1,\mathbf{0}}\right)$
(see Remark \ref{rem:retraction notation explained}) where 
\[
\Pi_{1,\mathbf{0}}:C\to HC,\; I_{1,\mathbf{0}}:HC\to C
\]
are chain maps: 
\[
\Pi_{1,\mathbf{0}}d'=0,d'I_{1,\mathbf{0}}=0,
\]
which satisfy 
\[
\Pi_{1,\mathbf{0}}I_{1,\mathbf{0}}=\mbox{id}_{HC},
\]
and $h_{1,\mathbf{0}}:C\to C\left[-1\right]$ is an $R$-module map
such that
\begin{gather*}
d'h_{1,\mathbf{0}}+h_{1,\mathbf{0}}d'x=I_{1,\mathbf{0}}\Pi_{1,\mathbf{0}}-\mbox{id}_{C}.
\end{gather*}
In addition, we require that the following \emph{side conditions }hold: 

\begin{align}
 & h_{1,\mathbf{0}}^{2}=0 & \Pi_{1,\mathbf{0}}h_{1,\mathbf{0}}=0 &  & h_{1,\mathbf{0}}I_{1,\mathbf{0}}=0\label{eq:side conditions}
\end{align}

\end{defn}

\begin{thm}
\label{thm:homological perturbation lemma}Let $\left(\bb\left(C^{G}\right),m\right)$
be a perturbation of $\left(C,d\right)$ and $\left(\Pi_{1,\mathbf{0}},I_{1,\mathbf{0}},h_{1,\mathbf{0}}\right)$
a retraction of $\left(C,d\right)$ to $HC$. Then there exists 
\begin{enumerate}
\item a perturbation $\left(\bb\left(\left(HC\right)^{G}\right),m^{can}\right)$
of $\left(HC,0\right)$, 
\item twisted $A_{\infty}$ morphisms 
\[
\Pi:\left(\bb\left(C^{G}\right),m\right)\to\left(\bb\left(\left(HC\right)^{G}\right),m^{can}\right)
\]
and 
\[
I:\left(\bb\left(\left(HC\right)^{G}\right),m^{can}\right)\to\left(\bb\left(C^{G}\right),m\right)
\]
with \emph{$\Pi\circ I=\mbox{id}_{\bb\left(\left(HC\right)^{G}\right)}$,}
and
\item a homotopy \emph{$h:\mbox{id}_{\bb\left(C^{G}\right)}\Rightarrow I\circ\Pi$}. 
\end{enumerate}
\end{thm}

\begin{rem}
\label{rem:retraction notation explained}Our somewhat odd notation
for the retraction is justified by the fact that the $\left(1,\mathbf{0}\right)$
component (see Remark \ref{rem:general coderivation in components})
of the coalgebra morphisms $\Pi$, $I$ and homotopy $h$ are the
retraction's $\Pi_{1,\mathbf{0}},I_{1,\mathbf{0}}$ and $h_{1,\mathbf{0}}$,
respectively.\end{rem}
\begin{proof}
Let $\left\{ m_{k,\beta}:C^{\otimes_{R}k}\to C\left[2-k-\mu\left(\beta\right)\right]\right\} $
denote the components of $m$ and let $\hat{m}_{k,\beta}:\left(C^{G}\left[1\right]\right)^{\hotimes_{\Lambda}k}\left[-1\right]\to C^{G}\left[1\right]$
denote the corresponding coderivations under the bijection (\ref{eq:coderivations in components}),
so $m=\sum_{k,\beta}\hat{m}_{k,\beta}$. Let $\partial'=m-\hat{m}_{1,\mathbf{0}}$. 

Set $\mathcal{P}=I_{1,\mathbf{0}}\circ\Pi_{1,\mathbf{0}}$. We define
morphisms $\hat{\Pi}_{1,\mathbf{0}}:\bb\left(C^{G}\right)\to\bb\left(HC^{G}\right)$
and $\hat{I}_{1,\mathbf{0}}:\bb\left(HC^{G}\right)\to\bb\left(C^{G}\right)$,
and an $\left(\mathcal{P},\mbox{id}_{C}\right)$-coderivation $\hat{h}_{1,\mathbf{0}}:\bb\left(C^{G}\right)\left[+1\right]\to\bb\left(C^{G}\right)$,
by setting

\begin{gather}
\hat{\Pi}_{1,\mathbf{0}}\left(x_{1},\hdots,x_{k}\right)=\Pi_{1,\mathbf{0}}\left(x_{1}\right)\otimes\cdots\otimes\Pi_{1,\mathbf{0}}\left(x_{k}\right),\nonumber \\
\hat{I}_{1,\mathbf{0}}\left(y_{1},\hdots,y_{k}\right)=I_{1,\mathbf{0}}\left(y_{1}\right)\otimes\cdots\otimes I_{1,\mathbf{0}}\left(y_{k}\right),\mbox{ and}\label{eq:extended linear retraction}\\
\hat{h}_{1,\mathbf{0}}\left(x_{1},\hdots,x_{k}\right)=\sum_{i=1}^{k}\left(-1\right)^{\sum_{j=1}^{i-1}\left(\codim x_{j}-1\right)}x_{1}\otimes\cdots\otimes x_{i-1}\otimes h_{1,\mathbf{0}}\, x_{i}\otimes\mathcal{P}x_{i+1}\otimes\cdots\otimes\mathcal{P}x_{k},\nonumber 
\end{gather}

\begin{gather*}
\\
\\
\end{gather*}
for $x_{1},...,x_{k}\in C^{G}$, $y_{1},...,y_{k}\in\left(HC\right)^{G}$,
cf. Eqs (\ref{eq:coderivations in components},\ref{eq:reconstructing a morphism}).
We claim the following expressions satisfy the conditions set out
in the theorem.

\begin{align}
m^{can} & =\sum_{a=0}^{\infty}\hat{\Pi}_{1,0}\partial'\left(\hat{h}_{1,0}\partial'\right)^{a}\hat{I}_{1,0}\label{eq:perturbed retraction}\\
\Pi & =\sum_{a=0}^{\infty}\hat{\Pi}_{1,0}\left(\partial'\hat{h}_{1,0}\right)^{a}\nonumber \\
I & =\sum_{a=0}^{\infty}\left(\hat{h}_{1,0}\partial'\right)^{a}\hat{I}_{1,0}\nonumber \\
h & =\sum_{a=0}^{\infty}\hat{h}_{1,0}\left(\partial'\hat{h}_{1,0}\right)^{a}.\nonumber 
\end{align}

Let us explain why the infinite sums converge point-wise. Suffice
it show that for every $\mathbf{x}\in\bb=\bb\left(C^{G}\right)$ and
$E\in\rr_{\geq0}$ there exists some $N$ with $\left(\partial'\hat{h}_{1,0}\right)^{a}\mathbf{x}\in\ff^{E}\bb$
for all $a\geq N$. Write $\mathbf{x}=x_{0}+x_{1}+\cdots+x_{k}+\cdots$
where $x_{k}\in\left(C^{G}\right)^{\hotimes_{\Lambda}k}$. We say
$x_{k}$ \emph{has length $k$.} 

There exists some $N'$ such that 

(a) $x_{k}\in\ff^{E}\bb\left(C^{G}\right)$ for all $k>N'$ and

(b) $\sum_{i=1}^{N'}E\left(\beta_{i}\right)\geq E$ for any sequence
of $N'$ elements $\beta_{i}\in G$, $\beta_{i}\neq0$ (see Def. \ref{def:discrete submonoid}).

Now we take $N=3N'$, and check that $\left(\partial'\hat{h}_{1,0}\right)^{a}\mathbf{x}\in\ff^{E}\bb$
for all $a\geq N$. Write $\partial'=\hat{m}_{\geq2,\mathbf{0}}+\hat{m}_{\bullet,\neq\mathbf{0}}$
where $\hat{m}_{\geq2,\mathbf{0}}=\sum_{k\geq2}\hat{m}_{k,\mathbf{0}}$
and $\hat{m}_{\bullet,\neq\mathbf{0}}=\sum_{\beta\neq0}\hat{m}_{k,\beta}$.
Since all the maps preserve the filtration, we find that $\left(\partial'\hat{h}_{1,0}\right)^{a}\mathbf{x}=\left(\hat{m}_{\geq2,\mathbf{0}}\hat{h}_{1,0}+\hat{m}_{\bullet,\neq\mathbf{0}}\hat{h}_{1,0}\right)^{a}\mathbf{x}_{\leq N'}\mod\ff^{E}\bb$,
where $\mathbf{x}_{\leq N'}:=x_{0}+x_{1}+\cdots+x_{N'}$. Expand $\left(\hat{m}_{\geq2,\mathbf{0}}\hat{h}_{1,0}+\hat{m}_{\bullet,\neq\mathbf{0}}\hat{h}_{1,0}\right)^{a}$
into a sum of $2^{a}$ products. Suppose one of these products has
$a_{+}$ factors of $\hat{m}_{\bullet,\neq\mathbf{0}}\hat{h}_{1,0}$,
which may increase the length, and $a_{-}$ factors of $\hat{m}_{\geq2,\mathbf{0}}\hat{h}_{1,0}$
which decrease the length, so $a_{-}\leq a_{+}+N'$ or else the product
vanishes on $\mathbf{x}_{\leq N'}$. Since we also have $a_{+}+a_{-}=a$
we find that for $a\geq N=3N'$ we have $a_{+}\geq N'$ which, by
condition (b), implies that $\left(\partial'\hat{h}_{1,0}\right)^{a}\mathbf{x}_{\leq N}\in\ff^{E}\bb$.
This completes the proof that (\ref{eq:perturbed retraction}) is
well-defined.

The verification that (\ref{eq:perturbed retraction}) satisfy all
the conditions of the theorem is a standard result, which we omit.
See the coalgebra perturbation lemma $\left(2.1_{*}\right)$ of \cite{Huebschmann+Kadeishvili}
for a very similar claim, with essentially the same proof. 
\end{proof}
We will explain the relation between this formalism and the more familiar
formalism of labeled ribbon trees, which is also used by Fukaya in
\cite{fukayacyclic}, in Remark \ref{rem:perturbed retraction as sum over trees}
below.

\subsection{Cyclic and unital homological algebra}

Next we want to discuss unital and cyclic refinements of Theorem \ref{thm:homological perturbation lemma}.

\begin{defn}
\label{def:cyclic unital retraction}Given a pairing $C\otimes_{R}C\to R\left[-p,q\right]$,
a retraction $\left(\Pi_{1,\mathbf{0}},I_{1,\mathbf{0}},h_{1,\mathbf{0}}\right)$
will be called \emph{cyclic }if
\begin{equation}
\left\langle h_{1,0}x,y\right\rangle +\left(-1\right)^{\codim x}\left\langle x,h_{1,0}y\right\rangle =0\label{eq:cyclic retraction}
\end{equation}
for any $x,y\in C$. An element $e\in C^{0,0}$ is \emph{a unit }for
the retraction $\left(\Pi_{1,\mathbf{0}},I_{1,\mathbf{0}},h_{1,\mathbf{0}}\right)$
if $h_{1,\mathbf{0}}e=0$. 

\end{defn}
\begin{prop}
\label{prop:unital and cyclic HPL}(a) Let $\left(\bb\left(C^{G}\right),m,\lozenge\right)$
be a cyclic twisted $A_{\infty}$ algebra, such that $\left(\bb\left(C^{G}\right),m\right)$
is a perturbation of $\left(C,d\right)$. Let $\left(\Pi_{1,\mathbf{0}},I_{1,\mathbf{0}},h_{1,\mathbf{0}}\right)$
be a cyclic unital retraction of $\left(C,d\right)$ to $HC$. Let
$I$ and $\Pi$ be given by Eq (\ref{eq:perturbed retraction}). Then
$m^{can}$ is cyclic with respect to the induced pairing on $HC$
and the morphism $I$ is cyclic. 

(b) Let $\left(\bb\left(C^{G}\right),m\right)$ be a twisted $A_{\infty}$
algebra with unit $\mathbf{e}$. If $\mathbf{e}$ is a unit for the
retraction $\left(\Pi_{1,\mathbf{0}},I_{1,\mathbf{0}},h_{1,\mathbf{0}}\right)$
then $\Pi_{1,\mathbf{0}}\mathbf{e}$ is a unit for $m^{can}$ and
the morphisms $I,\Pi$ given by Eq (\ref{eq:perturbed retraction})
are unital.\end{prop}
\begin{rem}
Let $f:\left(\bb,m,\lozenge\right)\to\left(\bb',m',\lozenge'\right)$
be a cyclic morphism. $f^{*}\lozenge'=\lozenge$ means, in particular,
that $\left\langle f_{1,\mathbf{0}}x,f_{1,\mathbf{0}}y\right\rangle '=\left\langle x,y\right\rangle $.
Since $\left\langle \cdot\right\rangle $ is non-degenerate, this
implies that $f_{1,0}$ must be injective. It follows that $\Pi$
will \emph{not} be cyclic, unless $m_{1,\mathbf{0}}=0$. Even so,
we should think of $\Pi_{1,\mathbf{0}}$ as inducing a cyclic equivalence
when a cyclic retraction is used. Cf. Remark 8.4 in \cite{fukayacyclic}.\end{rem}
\begin{proof}
[Proof of Proposition \ref{prop:unital and cyclic HPL}]We prove (a).
Let $\lozenge^{H}$ denote the pairing cocycle for $\bb\left(H\left(C^{G},m_{1,\mathbf{0}}\right)\right)$.
First we show $I$ is a cyclic morphism, i.e. that $I^{*}\lozenge=\lozenge^{H}$.
We choose some $x_{1},...,x_{k}\in C$, and compute:

\begin{multline*}
\left(I^{*}\lozenge\right)\left(x_{1},\cdots,x_{k}\right)=\\
=\sum_{a_{1},a_{2}\geq0}\sum_{k_{1}+k_{2}=k}\lozenge\left(\pi_{1}\left(\hat{h}_{1,0}\partial'\right)^{a_{1}}\hat{I}_{1,0}\left(...,x_{k_{1}}\right)\;\otimes\;\pi_{1}\left(\hat{h}_{1,0}\partial'\right)^{a_{2}}\hat{I}_{1,0}\left(x_{k_{1}+1},...\right)\right).
\end{multline*}
We now show that the summands with $a_{1}>0$ vanish. Indeed, for
some $\sigma\in C^{G}$ we can write

\begin{multline*}
\lozenge\left(h_{1,0}\sigma\otimes\left(\hat{h}_{1,0}\partial'\right)^{a_{2}}\hat{I}_{1,0}\left(x_{k_{1}+1}\otimes\cdots\otimes x_{k}\right)\right)=\\
=\pm\lozenge\left(\sigma\otimes h_{1,0}\pi_{1}\left(\hat{h}_{1,0}\partial'\right)^{a_{2}}\hat{I}_{1,0}\left(x_{k_{1}+1}\otimes\cdots\otimes x_{k}\right)\right)=0
\end{multline*}
where the last equality is by the side conditions (\ref{eq:side conditions}).
Terms with $a_{2}>0$ vanish for the same reason, which shows that
$I$ is cyclic. 

Now we show that $m^{can}$ is cyclic: 
\[
0=I^{*}m^{*}\lozenge=\left(m^{can}\right)^{*}I^{*}\lozenge=\left(m^{can}\right)^{*}\lozenge^{H}.
\]

We prove (b). Note that, since $m_{1,\mathbf{0}}\mathbf{e}=0$ and
$h_{1,\mathbf{0}}\mathbf{e}=0$, we have $I_{1,0}\Pi_{1,0}\mathbf{e}=\mathbf{e}$.
Let us check $\Pi$ is unital. Expand 
\begin{equation}
\pi_{1}\Pi=\pi_{1}\hat{\Pi}_{1,0}\sum\prod_{j=1}^{a}\left(m_{k_{j},\beta_{j}}^{\left(s_{j}\right)}\circ h_{1,0}^{\left(t_{j}\right)}\right)\label{eq:Pi expanded}
\end{equation}
where for $\mathbf{y}=y_{1}\otimes\cdots\otimes y_{l}$, $y_{i}\in C$,
the operators $m_{k,\beta}^{\left(s\right)},h_{1,\mathbf{0}}^{\left(t\right)}:\bb\left(C^{G}\right)\to\bb\left(C^{G}\right)$
(degree shift not shown) are given by 
\begin{multline*}
m_{k,\beta}^{\left(s\right)}\mathbf{y}=\\
\begin{cases}
\left(-1\right)^{\sum_{j=1}^{s}\left(\codim y_{j}-1\right)}\cdots y_{s}\otimes\lambda\left(\beta\right)m_{k,\beta}\left(y_{s+1},...,y_{s+k}\right)\otimes\cdots & \mbox{if }l\geq s+k\\
0 & \mbox{otherwise}
\end{cases}
\end{multline*}
and

\begin{multline*}
h_{1,\mathbf{0}}^{\left(t\right)}\mathbf{y}=\\
\begin{cases}
\left(-1\right)^{\sum_{j=1}^{s}\left(\codim y_{j}-1\right)}y_{1}\otimes\cdots\otimes y_{t}\otimes h_{1,\mathbf{0}}\left(y_{t+1}\right)\otimes\mathcal{P}y_{t+2}\otimes\cdots\otimes\mathcal{P}y_{l} & \mbox{if }l\geq s+1\\
0 & \mbox{otherwise}
\end{cases},
\end{multline*}
see (\ref{eq:extended linear retraction}). The sum in (\ref{eq:Pi expanded})
ranges over all $a\geq0$ and for each $a$ over $a$-tuples of quadruples
$\left(k_{j},\beta_{j},s_{j},t_{j}\right)$ with $k_{j},s_{j},t_{j}$
non-negative integers and $\beta_{j}\in G$. We use $\prod_{j=1}^{a}$
to denote the \emph{composition }of operators where the factor corresponding
to $j=1$ is applied\emph{ first.} 

We focus on the contribution of one summand of (\ref{eq:Pi expanded})
and consider 
\begin{equation}
\left(\pi_{1}\Pi_{1,\mathbf{0}}\prod_{j=1}^{a}\left(m_{k_{j},\beta_{j}}^{\left(s_{j}\right)}\circ\hat{h}_{1,0}^{\left(t_{j}\right)}\right)\right)\left(\mathbf{x}\right)\label{eq:Pi contribution}
\end{equation}
 for

\begin{equation}
\mathbf{x}=x_{1}\otimes\cdots\otimes\left(x_{r}=\mathbf{e}\right)\otimes\cdots\otimes x_{k}.\label{eq:input with a unit}
\end{equation}
Clearly, if $a=0$ we have a nonzero contribution only for $k=0$,
in which case we get $\Pi_{1,\mathbf{0}}\mathbf{e}$. We now show
all the other contributions either cancel in pairs or vanish. Clearly,
if $a\geq1$ and $\left(k_{j},\beta_{j}\right)\neq\left(2,\mathbf{0}\right)$
for all $1\leq j\leq a$ then (\ref{eq:Pi contribution}). More precisely,
let $r_{0}=r$ and, as long as $r_{j}$ is not in $\left\{ s_{j}+1,...,s_{j}+k_{j}\right\} $
define $r_{j+1}=\begin{cases}
r_{j} & \mbox{if }r_{j}\leq s_{j}\\
r_{j}-k_{j}+1 & \mbox{if }r_{j}>s_{j}+k_{j}
\end{cases}$. Note that if $r_{j}\not\in\left\{ s_{j}+1,...,s_{j}+k_{j}\right\} $
for all $1\leq j\leq a$ then (\ref{eq:Pi contribution}) vanishes
since $\Pi_{1,\mathbf{0}}\prod_{j=1}^{a}\left(m_{k_{j},\beta_{j}}^{\left(s_{j}\right)}\circ\hat{h}_{1,0}^{\left(t_{j}\right)}\right)$
has length $\geq2$. Suppose, then, that we have some $b\leq a$ with
$r_{b}\in\left\{ s_{j}+1,...,s_{j}+k_{j}\right\} $. To have a non-vanishing
contribution we must have $\left(k_{b},\beta_{b}\right)=\left(2,\mathbf{0}\right)$
and $s_{b}=r_{b}-1$ or $s_{b}=r_{b}-2$. Let us clean up the indexing,
replacing $\mathbf{x}$ with $ $$\prod_{j=1}^{b-1}\left(m_{k_{j},\beta_{j}}^{\left(s_{j}\right)}\circ\hat{h}_{1,0}^{\left(t_{j}\right)}\right)\mathbf{x}$
and $r$ with $r_{b}$, so $\mathbf{x}$ is still of the form (\ref{eq:input with a unit}).
We update $l$ accordingly and replace ${\pi_{1}\Pi_{1,\mathbf{0}}\prod_{j=1}^{a}\left(m_{k_{j},\beta_{j}}^{\left(s_{j}\right)}\circ\hat{h}_{1,0}^{\left(t_{j}\right)}\right)}$
with ${\pi_{1}\Pi_{1,\mathbf{0}}\prod_{j=1}^{a-b}\left(m_{k_{j+\left(b-1\right)},\beta_{j-\left(b-1\right)}}^{\left(s_{j+\left(b-1\right)}\right)}\circ\hat{h}_{1,0}^{\left(t_{j+\left(b-1\right)}\right)}\right)}$,
so now $\left(k_{1},\beta_{1}\right)=\left(2,\mathbf{0}\right)$ and
$s_{1}=r-1$ or $s_{1}=r-2$. We call this process \emph{bringing
$\mathbf{e}$ to the front. }

Note that if we had $r\neq1,k$, then the same is true after we bring
$\mathbf{e}$ to the front. In this case we have 

\begin{multline*}
m_{2,\mathbf{0}}^{\left(r-2\right)}\mathbf{x}+m_{2,\mathbf{0}}^{\left(r-1\right)}\mathbf{x}=\\
\left(-1\right)^{\sum_{j\leq r-2}\left(\codim x_{j}-1\right)}x_{1}\otimes\cdots\otimes m_{2,0}\left(x_{r-1}\otimes\mathbf{e}\right)\otimes x_{r+1}\otimes\cdots\otimes x_{l}\;+\\
+\;\left(-1\right)^{\sum_{j\leq r-1}\left(\codim x_{j}-1\right)}x_{1}\otimes\cdots\otimes x_{r-1}\otimes m_{2,0}\left(\mathbf{e}\otimes x_{r+1}\right)\otimes\cdots\otimes x_{l}=0,
\end{multline*}
which implies that contributions with $i\neq1,l$ can be canceled
in pairs. We complete the proof by showing contributions with $i=1$
or $i=l$ vanish. This involves checking a few cases as follows. We
again assume that we've brought $\mathbf{e}$ to the front, to simplify
the indices.

\textbf{Case 1, $i=1,s_{1}=0,a\geq2$: }We show $h_{1,\mathbf{0}}^{\left(t_{2}\right)}m_{2,\mathbf{0}}^{\left(0\right)}h_{1,\mathbf{0}}^{\left(t_{1}\right)}\mathbf{x}=0$.
If $t_{1}=t_{2}-1$ this is true because $h_{1,\mathbf{0}}^{2}=0$.
If $t_{1}=0$ we use $h_{1,\mathbf{0}}\mathbf{e}=0$. If $1\leq t_{1}\leq t_{2}-2$
we use $h_{1,\mathbf{0}}\mathcal{P}=0$. If $t_{1}\geq t_{2}$ we
use $\mathcal{P}h_{1,\mathbf{0}}=0$. 

\textbf{Case 2, $i=1,s_{1}=0,a=1$: }In this case we must have $k=2$;
$\pi_{1}\Pi_{1,0}m_{2,\mathbf{0}}^{\left(0\right)}h_{1,\mathbf{0}}^{\left(t_{1}\right)}\mathbf{x}=0$
for $t_{1}=0$ because $h_{1,\mathbf{0}}\mathbf{e}=0$ and for $t_{1}=1$
because $\Pi_{1,\mathbf{0}}h_{1,\mathbf{0}}=0$. 

\textbf{Case 3, $i=l,s_{1}=l-1,a\geq2$, and case 4, $i=l,s_{1}=l-1,a=1$:
}These are similar to the previous two cases. 

The proof that $\Pi_{1,\mathbf{0}}\mathbf{e}$ is a unit for $m^{can}$
and that $I$ is a unital morphism follow the same arguments and are
omitted. \end{proof}
\begin{rem}
\label{rem:perturbed retraction as sum over trees}By Proposition
\ref{prop:Differentials are twisted A8 algebras} the differential
\[
m^{can}=\sum_{a=0}^{\infty}\hat{\Pi}_{1,0}\partial'\left(\hat{h}_{1,0}\partial'\right)^{a}\hat{I}_{1,0}
\]
corresponds to some set of maps $\left\{ m_{k,\beta}^{can}\right\} $.
We have 
\[
m_{k,\beta}^{can}=\sum_{\Gamma\in Gr\left(k,\beta\right)}m_{\Gamma}
\]
where $Gr\left(k,\beta\right)$ is as in%
\footnote{There is a typo there, and $Gr\left(k,\beta\right)$ is written $Gr\left(\beta,k\right)$.
Compare with Eq (117) ibid.%
} Definition 9.1 of \cite{fukayacyclic}: it is a set of isomorphism
types of $G$-labeled rooted ribbon trees, which are represented by
triples $\left(T,v_{0},\beta\left(\cdot\right)\right)$ satisfying
conditions (1)-(4) at the beginning of Section 9 ibid. and Definition
9.1 (the only change is that for our discrete submonoid $G$, which
appears in condition (3), we do not assume $\mu\left(G\right)\subset2\zz$).
Note this equation is formally identical to Eq (117) of \cite{fukayacyclic}.
Indeed for $\mu\left(G\right)\subset2\zz$ the contribution $m_{\Gamma}$
is also precisely as described there. Let us briefly explain how to
see this.

Expand 
\begin{equation}
\pi_{1}m^{can}=\sum\pi_{1}\hat{\Pi}_{1,0}\left(m_{k_{a+1},\beta_{a+1}}^{\left(s_{a+1}\right)}\prod_{j=1}^{a}h_{1,\mathbf{0}}^{\left(t_{j}\right)}m_{k_{j},\beta_{j}}^{\left(s_{j}\right)}\right)\hat{I}_{1,\mathbf{0}}\label{eq:expansion of minimal model}
\end{equation}
(see the proof of Proposition \ref{prop:unital and cyclic HPL} for
an explanation of this notation) and apply each summand to $\mathbf{y}=y_{1}\otimes\cdots\otimes y_{k}$.

It is not hard to see that the side conditions (\ref{eq:side conditions})
imply that $t_{j}=s_{j}$, or else the contribution vanishes (so we
must apply each $h_{1,\mathbf{0}}$ to the output of previous $m_{k,\beta}$).
Moreover, since $\mathcal{P}h_{1,\mathbf{0}}=0$ we find that the
sequence $l_{j}-k_{j}-t_{j}$ is non-increasing (so each operation
is applied either after, or to the right of, the previous operations).
Here $l_{j}$ is the length of $\prod_{i=1}^{j}h_{1,\mathbf{0}}^{\left(t_{i}\right)}m_{k_{i},\beta_{i}}^{\left(s_{i}\right)}\mathbf{y}$,
with $l_{0}=l$. This defines an obvious map from the set of (generically)
nonzero summands in Eq (\ref{eq:expansion of minimal model}) to $G$-labeled
ribbon rooted trees.

There is a map in the other direction: embed $T$ in $\rr^{2}$ so
that $v_{0}$ has maximal height w.r.t. some linear projection $\rr^{2}\to\rr$.
Order the vertices of $T$ so that the vertices most distant from
$v_{0}$ (or ``deepest'') appear first, and vertices with the same
distance are ordered from left to right (using the planar embedding).
This determines a summand of Eq (\ref{eq:expansion of minimal model})
in an obvious way, where the order of the operations is given by the
order of the vertices of $T$. It is easy to check that these maps
define a bijection between isomorphism types of $G$-labeled rooted
ribbon trees $\Gamma$ and (generically) non-vanishing contributions
to Eq (\ref{eq:expansion of minimal model}). If $\mu\left(G\right)\subset2\zz$
this contribution is precisely $m_{\Gamma}$ of Eq (117) in \cite{fukayacyclic}. 

By Proposition \ref{prop:morphism in components} the morphism $I$
can also be decomposed as some $\left\{ I_{k,\beta}\right\} $, and
we have 
\[
I_{k,\beta}=\sum_{\Gamma\in Gr\left(k,\beta\right)}f_{\Gamma}
\]
for some $f_{\Gamma}$ which, in case $\mu\left(G\right)\subset2\zz$
are identical to $f_{\Gamma}$ of Eq (117) in \cite{fukayacyclic}.
The reasoning is essentially the same.
\end{rem}

\subsection{Constructing retractions in simple cases}

If $\left(C,d\right)$ is a dg $R$-module where $R$ is a field,
constructing a retraction of $\left(C,d\right)$ to $HC$ is a simple
matter; by choosing a splitting for the short exact sequences 
\[
0\to\ker d\to C\to\im d\to0
\]
and 
\[
0\to\im d\to\ker d\to HC\to0
\]
we obtain an internal direct sum decomposition
\begin{equation}
C=H\oplus d\left(C\right)\oplus O\label{eq:decompose C into three}
\end{equation}
where $ $$H\simeq HC$ and $d\left(C\right)\overset{\phi}{\simeq}O$
so that in these components the differential $d$ is given by 
\[
\begin{pmatrix}0 & 0 & 0\\
0 & 0 & \phi\\
0 & 0 & 0
\end{pmatrix},
\]
and then we can take $I_{1,\mathbf{0}}:H\to C$ and $\Pi_{1,\mathbf{0}}:C\to H$
to be the structure maps associated with the decomposition (\ref{eq:decompose C into three}),
and 
\[
h_{1,\mathbf{0}}=\begin{pmatrix}0 & 0 & 0\\
0 & 0 & 0\\
0 & \phi^{-1} & 0
\end{pmatrix}.
\]

Given any $\mathbf{e}\in C$ with $d\mathbf{e}=0$, it is easy to
modify the above construction so that $\mathbf{e}\in H\subset C$
in (\ref{eq:decompose C into three}), and then the retraction is
unital.

Given an antisymmetric, non-degenerate pairing $\left\langle \cdot\right\rangle :C\otimes_{R}C\to R\left[-p,q\right]$,
if the decomposition (\ref{eq:decompose C into three}) is an orthogonal
direct sum decomposition, then the associated retraction is cyclic. 

In the situation of Example \ref{exa:De Rham DGA} we can use the
Hodge-De Rham decomposition, see the discussion in \cite{fukayacyclic}
of the orientable case, the non-orientable case is the same.

\section{\label{sec:Equivariant Cohomology}Equivariant cohomology}

In this section we introduce the Cartan-Weil complex associated with
an equivariant manifold, and explain how to extend various constructions
to this equivariant setting. Throughout this section, we will not
use the $\grading$-grading. Some objects will be $\zz$-graded, and
local systems will be explicitly indicated.

\subsection{\label{sub:basic definitions}Basic definitions}

Fix a positive integer $n'$ and let $\mathbb{T}=\rr^{n'}/\zz^{n'}$
denote the torus group.  

Suppose now $X$ is a smooth $\mathbb{T}$-manifold, possibly with
boundary. Let $\left\{ \xi_{j}\right\} _{j=1}^{n'}$ denote the $n'$
commuting vector fields on $X$ with 1-periodic flows, which correspond
to the standard basis of $\rr^{n'}=\mbox{Lie}\left(\tb\right)$ under
the $\tb$-action. Suppose $\lc$ is a $\tb$-equivariant local system
on $X$. We now introduce \emph{the Cartan-Weil complex} $C^{CW}\left(X;\lc\right)$
which is a graded deformation of the usual De Rham complex of $\lc$-valued
differential forms on $X$. We will adopt the viewpoint that the cohomology
of the Cartan-Weil complex \emph{is }the equivariant cohomology of
$X$, though it should really be seen as a way to compute the cohomology
of the homotopy quotient of $X$ by the $\tb$ action. We refer the
reader to \cite{AB} and \cite{supersymmetry+equivariance} for more
comprehensive discussions of this rich subject. Our treatment will
be focused on the applications we have in mind.

We denote by $\lc\left[\alpha_{1},...,\alpha_{n'}\right]:=\lc\otimes_{\cc}\cc\left[\alpha_{1},...,\alpha_{n'}\right]$
the locally constant sheaf of $\zz$-graded abelian groups where $\deg\alpha_{i}=2$.
We will often denote $\cc\left[\vec{\alpha}\right]=\cc\left[\alpha_{1},...,\alpha_{n'}\right]$,
similarly for $\lc\left[\vec{\alpha}\right]$, etc. As an $R$-module,
the Cartan-Weil complex is 
\[
C^{CW}\left(X;\lc\right)=\Omega\left(X;\lc\left[\vec{\alpha}\right]\right)^{\mathbb{T}},
\]
the abelian group of $\tb$-invariant differential forms on $X$ with
values in $\lc\left[\vec{\alpha}\right]$. It is $\zz$-graded by
the sum of the De Rham and algebraic grading of $\lc\left[\vec{\alpha}\right]$.
We equip it with the degree one $\cc\left[\vec{\alpha}\right]$-linear
differential $D:C^{CW}\left(C;\lc\right)\to C^{CW}\left(C;\lc\right)\left[1\right]$
given by  
\[
D=d-\sum_{j=1}^{n'}\alpha_{j}\iota_{\xi_{j}},
\]
where $\iota_{\xi_{j}}$ denotes contraction with the vector field
$\xi_{j}$ on $X$. It is easy to see that $D^{2}=0$ on $\tb$-invariant
forms. If $\lc_{1},\lc_{2}$ are two equivariant local systems on
$X$ then the usual wedge product of forms preserves $\tb$-invariance
and so defines a map
\[
\wedge:C^{CW}\left(X;\lc_{1}\right)\otimes_{R}C^{CW}\left(X;\lc_{2}\right)\to C^{CW}\left(X;\lc_{1}\otimes_{\cc}\lc_{2}\right)
\]
which satisfies the graded Leibniz rule and is associative in the
obvious sense. It corresponds to the cup-product in equivariant cohomology.
If $f:Y\to X$ is a $\tb$-equivariant map the usual pullback map
of forms respects the invariant forms so we obtain a map 
\[
f^{*}:C^{CW}\left(Y;\lc\right)\to C^{CW}\left(X;f^{*}\lc\right).
\]
If ${\left(f:X\to Y,\kappa:Or\left(TX\right)\otimes f^{*}Or\left(TY\right)^{\vee}\to\kk^{\vee}\otimes f^{*}\lc\right)}$
is an oriented proper submersion (see Definition \ref{def:oriented submersion})
where $\kk$ and $\lc$ are \emph{$\tb$-equivariant} local systems
on $X$ and $Y$, respectively, then have a pushforward map 
\[
f_{!}^{\kappa}:C^{CW}\left(X;\kk\right)\to C^{CW}\left(Y;\lc\right)\left[\dim Y-\dim X\right]
\]
defined by integration on the fiber. See the appendix \ref{sec:Appendix.-Orientation-convention}
for our conventions regarding orientations and the pushforward. An
important special case is pushforward to a point, which is defined
when $X$ is compact:
\begin{equation}
\int:C^{CW}\left(X;Or\left(TX\right)^{\vee}\right)\to R\left[-\dim X\right].\label{eq:equivariant integration}
\end{equation}

\begin{rem}
\label{rem:ls equivariant is a property}If $X$ is a topological
space with a $\tb$-action, and $\lc$ is a local system on $X$,
then being $\tb$-equivariant is a \emph{property} and does not involve
additional choices (unlike lifting the action to a vector bundle over
$X$, for example). To see this, consider the maps $s,t:\tb\times X\to X$
where $s$ is the projection and $t$ is the action map. By definition,
if $\lc$ is $\tb$-equivariant there exists an isomorphism $\phi:s^{*}\lc\simeq t^{*}\lc$,
which we can think of as a nonzero global section $\tilde{\phi}$
of $\mathbf{Local}_{\tb\times X}\left(s^{*}\lc,t^{*}\lc\right)$,
the internal hom object in the category of local systems on $\tb\times X$.
$\phi$ is required to satisfy the identity axiom, which specifies
the value of $\tilde{\phi}$ at $\mathbf{e}\times X$. Since $\mathbf{e}\times X$
intersects every connected component of $\tb\times X$ we find that
this specifies the value of $\tilde{\phi}$ everywhere. It follows
that a $\tb$-action, if it exists, is unique. 

This also explains why we did not need to assume that the local system
isomorphism $\kappa$ is $\tb$-equivariant: \emph{any} isomorphism
between $\tb$-equivariant local systems $\lc_{1},\lc_{2}$ on $X$
is automatically $\tb$-equivariant. 
\end{rem}

\subsection{Equivariant cohomology in the separated case}
\begin{defn}
\label{def:separated}Let $\lc$ be a local system on a manifold $X$.
$X$ will be called \emph{$\lc$-separated} if $H^{\bullet}\left(X;\lc\right)$
is concentrated either only in even or only in odd degrees.\end{defn}
\begin{prop}
\label{prop:equivariant cohomology of separated manifolds}If $X$
is a $\tb$-manifold and $\lc$ is an equivariant local system on
$X$ such that $X$ is $\lc$-separated then there's an isomorphism
of $\ca$-modules $\Phi:\ca\otimes H\left(X;\lc\right)\simeq H^{CW}\left(X;\lc\right)$\end{prop}
\begin{proof}
Let $\mathfrak{m}\triangleleft\cc\left[\vec{\alpha}\right]$ denote
the maximal ideal generated by $\alpha_{1},...,\alpha_{n'}$. Consider
the filtration of $C^{CW}\left(X;\lc\right)$ by the dg $\cc\left[\vec{\alpha}\right]$-submodules
\[
C^{CW}\left(L\right)=\mf^{0}C^{CW}\left(X;\lc\right)\supset\mf^{1}C^{CW}\left(X;\lc\right)\supset\cdots
\]
Let $\left(E_{1}=\bigoplus_{i\geq0}\mf^{i}C^{CW}\left(X;\lc\right)/\mf^{i+1}C^{CW}\left(X;\lc\right),d_{1}\right),\,\left(E_{2},d_{2}\right),\,...$
denote the corresponding spectral sequence. All the terms are $\zz$-graded
$\cc\left[\vec{\alpha}\right]$-modules, where the $\cc\left[\vec{\alpha}\right]$-linear
differential $d_{i}$ has degree $1$. We have $E_{1}\simeq\ca\otimes C\left(L;\lc\right)^{\tb}$
and $d_{1}=\mbox{id}\otimes d$, so $E_{2}\simeq\cc\left[\vec{\alpha}\right]\otimes H\left(X;\lc\right)$,
which is concentrated either in even or odd (total) degree. It follows
that $d_{2},d_{3},...$ must all vanish and the spectral sequence
degenerates at the second page. This means that there's an isomorphism
of $\cc\left[\vec{\alpha}\right]$-modules 
\begin{equation}
\cc\left[\vec{\alpha}\right]\otimes H\left(X;\lc\right)\simeq\bigoplus_{i\geq0}\mf^{i}H^{CW}\left(X;\lc\right)/\mf^{i+1}H^{CW}\left(X;\lc\right).\label{eq:associated graded isomorphism-1}
\end{equation}
We now show that there's an isomorphism of $\ca$-modules ${\ca\otimes H\left(X;\lc\right)\simeq H^{CW}\left(X;\lc\right)}$.
It follows from Eq (\ref{eq:associated graded isomorphism-1}) that
the map $q:H^{CW}\left(X;\lc\right)\to H\left(X;\lc\right)$ obtained
by setting $\alpha_{j}=0$ for $1\leq j\leq n'$, is surjective. By
choosing a basis for $H\left(X;\lc\right)$, choosing $q$-preimages
in $H^{CW}\left(X;\lc\right)$, and then choosing equivariant $D$-closed
forms representing the corresponding cohomology classes, we define
a $\cc$-linear map of complexes ${\sigma:\left(H\left(X;\lc\right),0\right)\to\left(C^{CW}\left(X;\lc\right),D\right)}$.
Let ${\mbox{id}_{\ca}\otimes\sigma:\left(\ca\otimes H\left(X;\lc\right),0\right)\to\left(C^{CW}\left(X;\lc\right),D\right)}$
denote the extension of scalars. We claim 
\[
\Phi:=H\left(\mbox{id}_{\ca}\otimes\sigma\right):\ca\otimes H\left(X;\lc\right)\to H^{CW}\left(X;\lc\right)
\]
is an isomorphism. Since it is a $\ca$-module map it respects the
$\mathfrak{m}$-adic filtrations, and so (since $\bigcap_{i\geq0}\mf^{i}=0$
and the filtration on $C^{CW}\left(X;\lc\right)$ is complete at every
degree) it suffices to check that the induced map of the associated
graded modules 
\[
\bigoplus_{i\geq0}\left(\mf^{i}\otimes H\left(X;\lc\right)\right)/\left(\mf^{i+1}\otimes H\left(X;\lc\right)\right)\to\bigoplus_{i\geq0}\mf^{i}H^{CW}\left(X;\lc\right)/\mf^{i+1}H^{CW}\left(X;\lc\right)
\]
is an isomorphism. For this it suffices to check that the induced
map on the $E_{2}$ page is an isomorphism which is obvious by construction.\end{proof}
\begin{cor}
\label{cor:equivariant Kunneth}(equivariant Künneth formula) Suppose
$X_{1},X_{2}$ are $\tb$-manifolds and for $i=1,2$ we have an equivariant
local system $\lc_{i}$ on $X_{i}$ such that $X_{i}$ is $\lc_{i}$-separated.
Then we have the following isomorphisms of $\ca$-modules:

\begin{multline*}
H^{CW}\left(X_{1}\times X_{2};\lc_{1}\boxtimes\lc_{2}\right)\simeq\ca\otimes H\left(X_{1};\lc_{1}\right)\otimes H\left(X_{2};\lc_{2}\right)\simeq\\
\simeq H^{CW}\left(X_{1};\lc_{1}\right)\otimes_{\ca}H^{CW}\left(X_{2};\lc_{2}\right)
\end{multline*}

\end{cor}
We use the notation $\lc_{1}\boxtimes\lc_{2}:=pr_{1}^{*}\lc_{1}\otimes pr_{2}^{*}\lc_{2}$
for $pr_{i}:X_{1}\times X_{2}\to X_{i}$ the projection.
\begin{proof}
We find that $X_{1}\times X_{2}$ is $\lc_{1}\boxtimes\lc_{2}$ separated,
by the usual Künneth formula. So by Proposition \ref{prop:equivariant cohomology of separated manifolds}
and the Künneth formula again we have isomorphisms of $\ca$-modules:
\begin{multline*}
H^{CW}\left(X_{1}\times X_{2};\lc_{1}\boxtimes\lc_{2}\right)\simeq\ca\otimes H\left(X_{1}\times X_{1};\lc_{1}\boxtimes\lc_{2}\right)\simeq\\
\ca\otimes H\left(X_{1};\lc_{1}\right)\otimes H\left(X_{2};\lc_{2}\right)\simeq\left(\ca\otimes H\left(X_{1};\lc_{1}\right)\right)\otimes_{\ca}\left(\ca\otimes H\left(X_{2};\lc_{2}\right)\right)\simeq\\
\simeq H^{CW}\left(X_{1};\lc_{1}\right)\otimes_{\ca}H^{CW}\left(X_{2};\lc_{2}\right)
\end{multline*}

\end{proof}

\subsection{\label{sub:equiv ang and euler forms}Equivariant angular and Euler
forms}

The remainder of this section is devoted to extending some ideas from
Bott and Tu's book \cite{bott+tu} to the equivariant setting. 

Let $E\overset{\pi}{\longrightarrow}B$ be a smooth $\mathbb{T}$-equivariant
rank $n$ vector bundle. Recall $Or\left(E\right)$ denotes the corresponding
equivariant orientation local system on $B$ (see $\S$\ref{sub:conventions}).
Fix an auxiliary $\tb$-invariant bundle metric, and denote by $i_{S}:S\left(E\right)\hookrightarrow E$
the $\tb$-submanifold consisting of unit vectors in $E$. The map
$\pi_{S}=\pi\circ i_{S}:S\left(E\right)\to B$ makes this a $\tb$-equivariant
sphere bundle over $B$. Whenever we discuss the sphere bundle associated
with a vector bundle, we'll assume that an invariant metric is fixed
on $E$. The choice of metric is immaterial: if $\phi:E\simeq E'$
is an equivariant isomorphism of vector-bundles and we fix some invariants
metrics on $E$ and $E'$, then there exists an equivariant isometry
$\tilde{\phi}:E\to E'$ (``normalize'' $\phi$).

We define an equivariant local system isomorphism 
\[
{\kappa:Or\left(TS\left(E\right)\right)\otimes\pi_{S}^{*}Or\left(TB\right)^{\vee}\simeq\left(\pi_{S}^{*}Or\left(E\right)\right)}.
\]
Let $\underline{\rr}\to i_{S}^{*}TE$ be the bundle map corresponding
to the outward normal vector, where $\underline{\rr}$ denotes the
trivial line bundle on $S\left(E\right)$. We obtain equivariant
short exact sequences of vector bundles:

\begin{gather}
0\to i_{S}^{*}\pi^{*}E\to i_{S}^{*}TE\to i_{S}^{*}\pi^{*}TB\to0\nonumber \\
0\to\underline{\rr}\to i_{S}^{*}TE\to TS\left(E\right)\to0\label{eq:orienting sphere}
\end{gather}
where in the second line we map $\underline{\rr}$ to an outward pointing
radial vector. These sequences define $\kappa$, see $\S$\ref{sub:2 of 3 orientation}.
\begin{defn}
\emph{\label{def:equivariant angular form}An equivariant angular
form for $E$} is a form ${\phi\in\Omega\left(S\left(E\right);\left(\pi_{S}^{*}Or\left(E\right)^{\vee}\right)\left[\vec{\alpha}\right]\right)^{\mathbb{T}}}$
of degree $\left(n-1\right)$ such that 

(a) $\pi_{S!}^{\kappa}\left(\phi\right)=1$ and

(b) $D\phi=-\pi^{*}e$ for some form $e\in\Omega\left(B;Or\left(E\right)^{\vee}\left[\vec{\alpha}\right]\right)^{\mathbb{T}}$. 
\end{defn}
The degree $n$ form $e$ is called \emph{the Euler form associated
with $\phi$}; since $\pi$ is submersive it is uniquely determined
by condition (b). See also pg.113 of \cite{bott+tu}.
\begin{prop}
\label{prop:equiv ang form exists}Let $E\overset{\pi}{\longrightarrow}B$
be a $\tb$-equivariant vector bundle. An equivariant angular form
for $\pi$ exists.\end{prop}
\begin{proof}
Consider the ansatz 

\begin{gather*}
\phi=\sum_{J\in\zz_{\geq0}^{n'}}\phi_{J}\alpha^{J}\qquad\mbox{and}\qquad e=\sum_{J\in\intcone}e_{J}\alpha^{J}
\end{gather*}
where 
\begin{equation}
\phi_{J}\in\Omega^{n-1-2\left|J\right|}\left(S\left(E\right);\pi_{S}^{*}Or\left(E\right)^{\vee}\right)^{\tb}\mbox{ and }e_{J}\in\Omega^{n-2\left|J\right|}\left(B;Or\left(E\right)^{\vee}\right)^{\tb}.\label{eq:angular and euler in components}
\end{equation}
By degree considerations, we see that condition (b) in Definition
\ref{def:equivariant angular form} is equivalent to requiring 
\[
\left(\mbox{b}'\right)\quad\pi_{S!}^{\kappa}\left(\phi_{\mathbf{0}}\right)=1.
\]
We define a relation $\left(\mbox{a}_{J}\right)$ on (\ref{eq:angular and euler in components})
for every $J=\left(j_{1},...,j_{n'}\right)\in\intcone$, by the following
equation.

\[
\left(\mbox{a}_{J}\right)\quad d\phi_{J}-\sum_{\left\{ 1\leq a\leq n'|j_{a}\geq1\right\} }\iota_{\xi_{a}}\phi_{J-\mathbf{e}_{a}}=-\pi_{S}^{*}e_{J}
\]
where $\mathbf{e}_{a}\in\intcone$ denotes the $a$-th standard unit
vector. Clearly, condition (a) in Definition \ref{def:equivariant angular form}
is equivalent to the logical conjunction $\bigwedge\left(\mbox{a}_{J}\right)$,
taken over all $J\in\intcone$. 

Bott and Tu prove that there exist $\phi_{\mathbf{0}}\in\Omega^{n-1}\left(S\left(E\right);Or\left(E\right)\right),e_{\mathbf{0}}\in\Omega^{n}\left(B;\cc\right)$
such that conditions $\left(\mbox{a}_{\mathbf{0}}\right)$ and $\left(\mbox{b}'\right)$
hold (see \cite{bott+tu}, pg. 121-122, for the orientable case; the
non-orientable case is also discussed in that book). We may assume
without loss of generality that $\phi_{\mathbf{0}}$ and $e_{\mathbf{0}}$
are $\tb$-invariant by averaging. That is, we can replace $\phi_{\mathbf{0}}\leftarrow\pi^{\tb}\phi_{\mathbf{0}}$
and $e_{\mathbf{0}}\leftarrow\pi^{\tb}e_{\mathbf{0}}$, where $\pi^{\mathbb{T}}:C\left(L\right)\to C\left(L\right)^{\mathbb{T}}$
is the projection operator defined by averaging out the $\mathbb{T}$-action
using the Haar measure on $\mathbb{T}$. It is easy to check that
$\pi^{\tb}$ commutes with $d$ as well as with equivariant pushforward
and pullback, and induces a homotopy equivalence $\left(C\left(L\right),d\right)\simeq\left(C\left(L\right)^{\mathbb{T}},d|_{C\left(L\right)^{\tb}}\right)$.
We will continue to use such averaging arguments to establish invariance
of all the forms we need, without explicit mention. 

We now proceed by induction on $t\geq1$. It will be convenient to
extend $\phi_{J}$ and $e_{J}$ by zeros to all of $\zz^{n'}$. We
will also define $\iota_{a}:=\iota_{\xi_{a}}$. With this we can rewrite
the relation $\left(\mbox{a}_{J}\right)$ as 
\[
\left(\mbox{a}_{J}\right)\quad d\phi_{J}-\sum_{J\in\zz^{n'}}\iota_{a}\phi_{J-\mathbf{e}_{a}}=-\pi_{S}^{*}e_{J}.
\]
Now, assume that $\phi_{J},e_{J}$ are given as in (\ref{eq:angular and euler in components})
for all $\left|J\right|\leq t$, such that $\left(\mbox{a}_{J}\right)$
holds for all $\left|J\right|\leq t$. We show that we can define
$\phi_{J},e_{J}$ for $\left\{ J:\left|J\right|=t+1\right\} $ so
that $\left(\mbox{a}_{J}\right)$ holds for such $J$ too. 

We will need the following part of the Gysin sequence for $S\left(E\right)$.\begin{center}
\begin{tikzpicture}[descr/.style={fill=white,inner sep=1.5pt}]
\matrix (m) [matrix of math nodes, row sep=1.6em, 
column sep=2.2em, text height=1.5ex, text depth=0.25ex]
{ & H^{n-2t-2}(B;Or(E)^\vee) & H^{n-2t-2}(S(E);Or(E)^\vee)&\\
H^{-2t-1}(B)=0 & H^{n-2t-1}(B;Or(E)^\vee) & H^{n-2t-1}(S(E);\pi_S^*Or(E)^\vee)\\
};

\path[overlay,->, font=\scriptsize,>=latex]         

(m-1-2) 
edge node[descr,yshift=0.3ex] {$\pi_{S}^*$}
(m-1-3)	

(m-1-3) 
edge[out=355,in=175] node[descr,yshift=0.3ex] {$\pi_{S!}^\kappa$} 
(m-2-1)

(m-2-1)edge(m-2-2)

(m-2-2) 
edge node[descr,yshift=0.3ex] {$\pi_{S}^*$}
(m-2-3);
\end{tikzpicture}
\end{center}

Focus on some $J=\left(j_{1},...,j_{n'}\right)$ with $\left|J\right|=t+1$.
We have

\begin{multline}
\pi_{S}^{*}d\left(\sum_{a=1}^{n'}\iota_{a}e_{J-\mathbf{e}_{a}}\right)=d\pi_{S}^{*}\left(\sum_{a=1}^{n'}\iota_{a}e_{J-\mathbf{e}_{a}}\right)=d\left(\sum_{a=1}^{n'}\iota_{a}\pi_{S}^{*}e_{J-\mathbf{e}_{a}}\right)=\\
=d\left(\sum_{a=1}^{n'}\iota_{a}\pi_{S}^{*}\left(-d\phi_{J-\mathbf{e}_{a}}+\sum_{b=1}^{n'}\iota_{b}\phi_{J-\mathbf{e}_{a}-\mathbf{e}_{b}}\right)\right)=\\
=\pi_{S}^{*}\left(\sum_{a=1}^{n'}d^{2}\iota_{a}\phi_{J-\mathbf{e}_{a}}+d\sum_{1\leq a,b\leq n'}\iota_{a}\iota_{b}\phi_{J-\mathbf{e}_{a}-\mathbf{e}_{b}}\right)=0.\label{eq:pullback of contraction of euler}
\end{multline}
In the last equality we used $d^{2}=0$ and $\sum_{1\leq a,b\leq n'}\iota_{a}\iota_{b}=0$.
Since $\pi_{S}$ is a submersion, $\pi_{S}^{*}$ is injective and
we conclude that $\sum_{a=1}^{n'}\iota_{a}e_{J-\mathbf{e}_{a}}$ is
$d$-closed and in the kernel of 
\[
\pi_{S}^{*}:H^{n-2t+1}\left(B;Or\left(E\right)^{\vee}\right)\to H^{n-2t+1}\left(S\left(E\right);\pi_{S}^{*}Or\left(E\right)^{\vee}\right).
\]
The second row of the Gysin sequence above shows that $\pi_{S}^{*}$
is injective on cohomology classes of degree $n-2t-1$ so we conclude
that $\sum_{a=1}^{n'}\iota_{a}e_{J-\mathbf{e}_{a}}=de_{J}'$ for some
$e'_{J}\in\Omega^{n-2\left(t+1\right)}\left(B;Or\left(E\right)^{\vee}\right)^{\tb}$. 

Next we compute

\begin{multline*}
d\left(-\sum_{a=1}^{n'}\iota_{a}\phi_{J-\mathbf{e}_{a}}+\pi_{S}^{*}e_{J}'\right)=\sum_{a=1}^{n'}\iota_{a}d\phi_{J-\mathbf{e}_{a}}+\pi_{S}^{*}de_{J}'=\\
=\sum_{a=1}^{n'}\iota_{a}\left(\sum_{b=1}^{n'}\iota_{b}\phi_{J-\mathbf{e}_{a}-\mathbf{e}_{b}}-\pi_{S}^{*}e_{J-\mathbf{e}_{a}}\right)+\pi_{S}^{*}de_{J}'=\pi_{S}^{*}\left(de_{J}'-\sum_{a=1}^{n'}\iota_{a}e_{J-\mathbf{e}_{a}}\right)=0
\end{multline*}
Using the top row of the Gysin sequence above we conclude  that
\[
\left[-\sum_{a=1}^{n'}\iota_{a}\phi_{J-\mathbf{e}_{a}}+\pi_{S}^{*}e_{J}'\right]\in\im\left(\pi_{S}^{*}\right).
\]
\sloppy In other words, there exists some ${x\in\Omega^{n-2t}\left(B;Or\left(E\right)^{\vee}\right)^{\tb}}$
and some ${y\in\Omega^{n-1-2t}\left(S\left(E\right);\pi_{S}^{*}Or\left(E\right)^{\vee}\right)^{\tb}}$
such that 
\[
-\sum_{a=1}^{n'}\iota_{a}\phi_{J-\mathbf{e}_{a}}+\pi_{S}^{*}e_{J}'=\pi_{S}^{*}x+dy.
\]
It follows that 
\begin{gather*}
\phi_{J}=-y\\
e_{J}=e_{J}'-x.
\end{gather*}
satisfy condition $\left(\mbox{a}_{J}\right)$. In this way we define
$\phi_{J},e_{J}$ for all $\left\{ J:\left|J\right|=t+1\right\} $,
establishing the inductive step. This concludes the proof of the proposition.
\end{proof}
We will not use the following proposition in this paper, but state
it for completeness and future reference. 
\begin{prop}
Let $E\overset{\pi}{\longrightarrow}B$ be an equivariant vector bundle.
Choose an equivariant angular form $\phi$ for $\pi$, with associated
equivariant Euler form $e$.

(a) $De=0$. Moreover, if $\tilde{\phi}$ is another equivariant angular
form for $\pi$ and $\tilde{e}$ is the associated equivariant Euler
form then there exists some form $\gamma\in\Omega\left(B;Or\left(E\right)^{\vee}\bva\right)^{\tb}$
such that 
\[
\tilde{e}-e=D\gamma.
\]

(b) Let $f:B'\to B$ be a smooth equivariant map and $E'\overset{\pi'}{\longrightarrow}B'$
the pullback bundle. If we use the pulled-back bundle metric on $E'$,
we obtain a Cartesian square 
\[
\xymatrix{S\left(E'\right)\ar[r]^{\tilde{f}}\ar[d] & S\left(E\right)\ar[d]\\
B'\ar[r]_{f} & B
}
\]
and $\tilde{f}^{*}\phi$ is an angular form for $E'$; the associated
Euler form is $f^{*}e$.\end{prop}
\begin{proof}
Part (a) is proved along the same lines as the previous proposition.
Part (b) is immediate.
\end{proof}

\subsection{General equivariant pushforward and Poincare duality}

We now want to define pushforward along an equivariant embedding.
Using the graph construction, this can be used to define pushforward
along any equivariant map. As a special case, we obtain a Poincare
dual to an equivariant submanifold.

\subsubsection{\label{sub:Blowing-up}Blowing up.}

Before we can define the equivariant pushforward, we need to discuss
a certain blow up construction that will allow us to construct the
Thom form for an equivariant bundle.
\begin{defn}
(a) Let $E\overset{\pi}{\longrightarrow}B$ be an equivariant vector
bundle. The \emph{blow up $\widetilde{E}$ of the zero section} \emph{of
$\pi$} is the $\tb$-manifold with boundary $\widetilde{E}:=S\left(E\right)\times[0,\infty)$,
where the action on $[0,\infty)$ is trivial. It comes equipped with
an equivariant smooth map $\beta:\widetilde{E}\to E$ so that $\left(v,r\right)\in S\left(E\right)\times[0,\infty)$
maps to $r\cdot i_{S\left(N\right)}\left(v\right)$ (cf. the discussion
at the beginning of $\S$\ref{sub:equiv ang and euler forms}). $\tilde{\pi}:=\pi\circ\beta:\widetilde{E}\to B$
turns this into an equivariant fiber-bundle.

(b) Let $X\overset{i}{\hookrightarrow}Y$ be an equivariant embedding
of closed $\tb$-manifolds. Let $N=N_{X/Y}$ denote the associated
equivariant normal bundle, and fix some equivariant tubular neighbourhood
$X\subset N\subset Y$ of $X$ in $Y$. We define \emph{the blow
up of $Y$ along $X$ }by 
\[
\bl_{X}Y:=\widetilde{N}\bigcup_{S\left(N\right)\times\left(0,\infty\right)}\left(Y\backslash X\right)
\]
where we use the open embeddings ${S\left(N\right)\times\left(0,\infty\right)\subset S\left(N\right)\times[0,\infty)}$
and ${S\left(N\right)\times\left(0,\infty\right)\overset{\beta|_{S\left(N\right)\times\left(0,\infty\right)}}{\longrightarrow}N\backslash\mbox{zero section}\subset Y\backslash X}$
to glue. The maps $\beta:S\left(N\right)\times[0,\infty)\to N\subset Y$
and $Y\backslash X\to Y$ glue to form a map 
\[
\beta_{X/Y}:\bl_{X}Y\to Y.
\]
We check $\bl_{X}Y$ is Hausdorff: if distinct points $x,y\in\bl_{X}Y$
satisfy ${\beta_{X/Y}\left(x\right)\neq\beta_{X/Y}\left(y\right)}$
then they can be separated in $Y$, otherwise they can be separated
in $S\left(N_{\beta_{X/Y}\left(x\right)}\right)$. It follows that
$\bl_{X}Y$ is a manifold with boundary. \end{defn}
\begin{rem}
Our terminology is non-standard. What is usually referred to as the
blow up of $Y$ along $X$ is obtained from the above construction
by identifying antipodal points in $S\left(N\right)\subset Y$ (if
we're doing real algebraic geometry) or $S^{1}$ orbits associated
with a $U\left(1\right)$ action (if we're working over $\cc$, so
in particular, $N$ can be taken to be a Hermitian vector bundle).
Either way, the usual construction produces a manifold without boundary.
On the other hand, $\bl_{X}Y$ is relatively orientable, see Lemma
\ref{lem:blow up is relatively orientable}.

The map $\beta_{X/Y}$ satisfies a natural universal property, which
in particular means the blow up is essentially unique, independent
of the choices (tubular neighbourhood, bundle metric) and explains
our use of the definite noun (\emph{the} blow up). We will not go
into details since for the purposes of this paper, we do not care
about the uniqueness of the construction.
\end{rem}
Let $\widetilde{E}=S\left(E\right)\times[0,\infty)$ be the blow up
of $E$. We denote by $pr_{S}:\widetilde{E}\to S\left(E\right)$ and
$pr_{r}:\widetilde{E}\to[0,\infty)$ the projections. 
\begin{lem}
\label{lem:blow up is relatively orientable}(a) There is a natural
equivariant local system isomorphism
\begin{equation}
Or\left(T\widetilde{E}\right)\simeq\beta^{*}Or\left(TE\right).\label{eq:blow up orientation}
\end{equation}
On the open set $U=\widetilde{E}\backslash\left(S\left(E\right)\times\left\{ 0\right\} \right)$
this isomorphism is induced by the isomorphism of vector bundles 
\[
d\beta:T\widetilde{E}\to TE.
\]

(b) If $X\overset{i}{\hookrightarrow}Y$ is an equivariant embedding
of closed $\tb$-manifolds, there's a canonical isomorphism of local
systems 

\[
Or\left(T\bl_{X}Y\right)\simeq\beta_{X/Y}^{*}Or\left(TY\right)
\]
which agrees with (\ref{eq:blow up orientation}) on $\widetilde{N}\subset\bl_{X}Y$.

(c) If $\im\left(i\right)\neq Y$ then 
\[
\int_{Y}\omega=\int_{\bl_{X}Y}\beta_{X/Y}^{*}\omega
\]
for any $\omega\in\Omega\left(Y;Or\left(TY\right)\bva^{\vee}\right)^{\tb}$,
where we use the isomorphism of part (b) to identify ${\Omega\left(Y;\beta^{*}Or\left(TY\right)\bva^{\vee}\right)^{\tb}=\Omega\left(Y;Or\left(T\bl_{X}Y\right)\bva^{\vee}\right)^{\tb}}$.\end{lem}
\begin{proof}
To construct (\ref{eq:blow up orientation}) use the ordered direct
sum decomposition

\[
\begin{gathered}T\widetilde{E}=pr_{S}^{*}TS\left(E\right)\oplus pr_{r}^{*}T[0,\infty)\end{gathered}
\]
and the short exact sequence 
\[
0\to TS\left(E\right)\to\pi_{S}^{*}TE\overset{dr}{\longrightarrow}\underline{\rr}\to0
\]
where $r$ is the distance from the zero section of $E$. $pr_{r}^{*}T[0,\infty)$
and $\underline{\rr}$ are both canonically oriented by the positive
direction, and so the decompositions define the desired isomorphism
by $\S$\ref{sub:2 of 3 orientation}. It is easy to see that on $U$
this isomorphism agrees with $d\beta$.

To prove (b), write $\bl_{X}Y=U_{1}\cup U_{2}$ where $U_{1}=\widetilde{N}$
and $U_{2}=Y\backslash X$. For $i=1,2$ we have isomorphisms $Or\left(T\bl_{X}Y\right)|_{U_{i}}\simeq\beta_{X/Y}^{*}Or\left(TY\right)|_{U_{i}}$
(for $i=1$ by part (a), for $i=2$ by the identity). These isomorphisms
agree on the intersection $U_{1}\cap U_{2}=U$. The result follows.

For part (c), note that if we restrict to the complement of an arbitrarily
small neighbourhood of $X\hookrightarrow Y$, the integrands can be
identified. The result follows.
\end{proof}

\subsubsection{The equivariant Thom form.}

Fix once and for all some smooth function $\sigma:[0,\infty)\to[-1,0]$
which satisfies $\sigma\left(r\right)=-1$ for $r\leq\frac{1}{2}$
and $\sigma\left(r\right)=0$ for $r\geq1$. Let $E\overset{\pi}{\longrightarrow}B$
be a rank $n$ equivariant vector bundle, let $\widetilde{E}=S\left(E\right)\times[0,\infty)$
be the blow up of $E$. Choose an equivariant angular form $\phi\in\Omega\left(S\left(E\right);Or\left(E\right)^{\vee}\bva\right)^{\tb}$
for $E$. There is a unique form $\tau_{E}\in\Omega^{n}\left(E;\pi^{*}Or\left(E\right)^{\vee}\bva\right)^{\tb}$
which satisfies 
\begin{equation}
\beta^{*}\tau_{E}=D\left(pr_{S}^{*}\phi\wedge pr_{r}^{*}\sigma\right).\label{eq:thom defining equation}
\end{equation}

\begin{defn}
$\tau_{E}$ is called \emph{an equivariant Thom form for $E$.}
\end{defn}
Note that $\tau_{E}$ is $D$-closed since it pulls back to a $D$-exact
form on the blow up.

\subsubsection{Pushforward along an equivariant embedding.}
\begin{defn}
\label{def:oriented embedding}(cf. Definition \ref{def:oriented submersion})
An \emph{oriented equivariant embedding }is a pair $\left(i,\kappa\right)$
where $i:X\hookrightarrow Y$ is a $\tb$-equivariant embedding of
closed $\tb$-manifolds and 
\[
\kappa:Or\left(TX\right)\otimes i^{*}Or\left(TY\right)^{\vee}\otimes\kk\simeq i^{*}\lc
\]
is an (equivariant, see Remark \ref{rem:ls equivariant is a property})
isomorphism of local systems, where $\kk,\lc$ are equivariant local
systems on $X$ and on $Y$, respectively. 
\end{defn}
Let $\left(i,\kappa\right)$ be an oriented equivariant embedding,
set $n=\dim Y-\dim X$, and let $N\overset{\pi}{\longrightarrow}X$
denote the rank $n$ normal bundle associated with $i$:
\begin{equation}
0\to TX\to i^{*}TY\to N\to0.\label{eq:normal bundle ses}
\end{equation}
Let $\tau_{N}\in\Omega\left(N;Or\left(TX\right)^{\vee}\otimes i^{*}Or\left(TY\right)\bva\right)^{\tb}$
be an equivariant Thom form for $N$, where we've used (\ref{eq:normal bundle ses})
to identify $Or\left(N\right)^{\vee}\simeq Or\left(TX\right)^{\vee}\otimes i^{*}Or\left(TY\right)$
(cf. $\S$\ref{sub:2 of 3 orientation}).
\begin{defn}
Given an oriented equivariant embedding $\left(i,\kappa\right)$ we
define an equivariant pushforward 
\[
i_{!}^{\kappa}:\Omega\left(X;\kk\bva\right)^{\tb}\to\Omega\left(Y;\lc\bva\right)^{\tb}\left[n\right]
\]
by 
\[
i_{!}^{\kappa}\left(\omega\right):=\tau_{N}\wedge\pi^{*}\omega,
\]
where we interpret the right hand side to mean the zero extension
(to $Y$) of the compactly supported form $\tau_{N}\wedge\pi^{*}\omega$
which is defined on $N$.\end{defn}
\begin{prop}
(Equivariant Poincare duality.) For any pair of $D$-closed forms
$\omega_{X}\in\Omega\left(X;\kk\bva\right)^{\tb}$, $\omega_{Y}\in\Omega\left(Y;\left(\lc\otimes Or\left(TY\right)\right)^{\vee}\bva\right)^{\tb}$,
we have 
\[
\int_{Y}i_{!}^{\kappa}\omega_{X}\wedge\omega_{Y}=\int_{X}\omega_{X}\wedge i^{*}\omega_{Y}
\]
where on the right we think of $\omega_{X}\wedge i^{*}\omega_{Y}$
as taking values in $Or\left(TX\right)^{\vee}$ using the identification
$\kk\otimes i^{*}\left(\lc\otimes Or\left(TY\right)\right)^{\vee}\simeq Or\left(TX\right)^{\vee}$
induced by $\kappa$.\end{prop}
\begin{proof}
Denote by ${i_{\partial}:S\left(N\right)=S\left(N\right)\times\left\{ 0\right\} \hookrightarrow\bl_{X}Y}$
the inclusion of the boundary. We compute
\begin{multline*}
\int_{Y}i_{!}^{\kappa}\omega_{X}\wedge\omega_{Y}\overset{\left(1\right)}{=}\int_{\widetilde{N}}D\left(pr_{S}^{*}\phi\wedge pr_{r}^{*}\sigma\right)\wedge\tilde{\pi}^{*}\omega_{X}\wedge\beta^{*}\omega_{Y}|_{\widetilde{N}}\overset{}{=}\\
\overset{}{=}\int_{\widetilde{N}}D\left(pr_{S}^{*}\phi\wedge pr_{r}^{*}\sigma\wedge\tilde{\pi}^{*}\omega_{X}\wedge\beta^{*}\omega_{Y}|_{\widetilde{N}}\right)\overset{\left(2\right)}{=}\\
=-\int_{S\left(N\right)}i_{\partial}^{*}\left(pr_{S}^{*}\phi\wedge pr_{r}^{*}\sigma\wedge\tilde{\pi}^{*}\omega_{X}\wedge\beta^{*}\omega_{Y}|_{\widetilde{N}}\right)=\\
\overset{}{=}-\int_{S\left(N\right)}\phi\wedge\sigma\left(0\right)\wedge\pi_{S}^{*}\left(\omega_{X}\wedge i^{*}\omega_{Y}\right)=\int_{X}\omega_{X}\wedge i^{*}\omega_{Y}.
\end{multline*}
In $\left(1\right)$ we've used part (c) of Lemma \ref{lem:blow up is relatively orientable},
and the fact the integrand is supported in $\widetilde{N}\subset\bl_{X}Y$.
In (2) we picked up a minus sign, since we think of $S\left(N\right)$
as oriented as in \ref{eq:orienting sphere}, using the outward pointing
radial vector. This is the same as thinking of $S\left(N\right)$
as the boundary of the associated disc bundle $D\left(N\right)$,
see $\S$\ref{sub:boundary orientation}, which is opposite its orientation
as the boundary of $\bl_{X}Y$, which is the orientation used to apply
Stokes' Theorem \ref{lem:Stokes'-Theorem.}.
\end{proof}

\begin{cor}
\label{cor:equivariant poincare duality}Given an equivariant isomorphism
of local systems 
\begin{equation}
{\kappa:Or\left(TX\right)\otimes i^{*}Or\left(TY\right)^{\vee}\simeq i^{*}\lc},\label{eq:ls iso needed}
\end{equation}
we call $\tau_{X/Y}:=i_{!}^{\kappa}1$ the \emph{($\lc$-valued) equivariant
Poincare dual to $X$ in $Y$}. We have 
\[
\int_{Y}\tau_{X/Y}\wedge\omega=\int_{X}\omega|_{X}
\]
for any $D$-closed form $\omega\in\Omega\left(Y;\left(\lc\otimes Or\left(TY\right)\right)^{\vee}\right)$.
\end{cor}

\section{\label{sec:Equivariant A8 algebras}Equivariant twisted $A_{\infty}$
algebras}

\subsection{The equivariant cyclic DGA, statement of results}

Let $L$ be an $n$-dimensional closed manifold equipped with a smooth
$\tb$-action. Since in our forthcoming papers we'll need to consider
non-orientable $L$, we will assume that $L$ is non-orientable. All
of the results hold, with the obvious changes, for $L$ oriented. 

We reintroduce the $\grading$-grading. We work over the unital commutative
real algebra $R=\cc\left[\vec{\alpha}\right]=\cc\left[\alpha_{1},...,\alpha_{n'}\right]$,
where $\deg\alpha_{i}=\left(2,0\right)$. Denote by $C^{CW}\left(L\right)$
the \emph{Cartan Weil $R$-DGA}:
\[
C^{CW}\left(L\right):=C^{CW}\left(L;\underline{\cc}\right)\oplus C^{CW}\left(L;Or\left(TL\right)\right)
\]
it is equipped with the differential $D$ and wedge product $\wedge$,
see $\S$\ref{sub:basic definitions}. It is $\grading$-graded. The
codimension $\zz$-component of the grading is as in $\S$\ref{sub:basic definitions}
and the $\zz/2$-component specifies the local system degree, as in
Example \ref{exa:De Rham DGA}.

We can think of $C^{CW}\left(L\right)$ as a deformation over $\ca$
of the $\cc$-DGA $C\left(L\right)$, in the sense of Remark \ref{rem:deformation explained},
as follows: take $G=G'=\left\{ 0\right\} $, the trivial gapping monoids,
$R=\cc\left[\vec{\alpha}\right]$, and $R'=\cc$. We have $\Lambda=\Lambda_{0}^{G}\left(R\right)=\ca$
and $\Lambda'=\Lambda_{0}^{G'}\left(R'\right)=\cc$, and there's an
obvious quotient map $\Lambda\overset{\vec{\alpha}=0}{\longrightarrow}\Lambda'$.
Let $\left(\bb^{CW},D+\wedge\right)$ and $\left(\bb,d+\wedge\right)$
denote the differential coalgebras corresponding to the DGA's $C^{CW}\left(L\right)$
and $C\left(L\right)$, as in Example \ref{exa:DGA} and Proposition
\ref{prop:Differentials are twisted A8 algebras} (the sloppy-but-visually-suggestive
notation $d+\wedge,D+\wedge$ for the differentials will not be used
outside this introduction). We then have $\left(d+\wedge\right)=\left(D+\wedge\right)\otimes_{R}\mbox{id}_{R'}$.

In the next subsection we'll present a simple condition for when a
deformation $\left(\bb,m\right)$ of $C\left(L\right)$ can be extended
to a deformation $\left(\bb^{CW},m^{CW}\right)$ of $C^{CW}\left(L\right)$.
In this case we obtain a commutative square of extensions
\begin{equation}
\xymatrix{\left(\bb_{\cc},d+\wedge\right) & \left(\bb_{\cc\left[\vec{\alpha}\right]}^{CW},D+\wedge\right)\ar[l]_{\vec{\alpha}=0}\\
\left(\bb_{\Lambda_{0}^{G}\left(\cc\right)},m\right)\ar[u]^{\tau=0} & \left(\bb_{\Lambda_{0}^{G}\left(\cc\left[\vec{\alpha}\right]\right)}^{CW},m^{CW}\right)\ar[l]_{\vec{\alpha}=0}\ar[u]^{\tau=0}
}
\label{eq:square of deformations}
\end{equation}

Things become more interesting when we consider cyclic symmetry. We
use integration $\int:C^{CW}\left(L\right)\to R\left[-n,1\right]$.
to define a pairing (cf. Eq (\ref{eq:pairing from functional})) ${\left\langle \cdot\right\rangle ^{CW}:C^{CW}\left(L\right)\otimes_{R}C^{CW}\left(L\right)\to C^{CW}\left(L\right)\left[-n,1\right]}$.

\begin{lem}
\label{lem:CW pairing is non-degenerate}$\left\langle \cdot\right\rangle ^{CW}$
is non-degenerate.\end{lem}
\begin{proof}
Suppose $0\neq\omega\in C^{CW}\left(L\right)$. Write $\omega=\sum_{I\in\zz_{\geq0}^{n'}}\omega_{I}\alpha^{I}$
with $\omega_{I}\in C\left(L\right)^{\mathbb{T}}$ and ${\alpha^{I}=\alpha^{\left(i_{1},...,i_{n'}\right)}:=\prod_{r=1}^{n'}\alpha_{r}^{i_{r}}}$.
Choose some $I_{0}$ so $\omega_{I_{0}}\neq0$. Since $\left\langle \cdot\right\rangle $
is non-degenerate, there's some $\omega'\in C\left(L\right)$ with
$\left\langle \omega_{I_{0}},\omega'\right\rangle \neq0$. We have
${0\neq\left\langle \omega_{I_{0}},\omega'\right\rangle =\left\langle \pi^{\tb}\omega_{I_{0}},\pi^{\tb}\omega'\right\rangle =\left\langle \omega_{I_{0}},\pi^{\tb}\omega'\right\rangle }$
(see the proof of Proposition \ref{prop:equiv ang form exists} for
the definition and some properties of $\pi^{\tb}$) so we may assume
without loss of generality that $\eta\in C\left(L\right)^{\mathbb{T}}\subset C^{CW}\left(L\right)$.
Since the coefficient of $\alpha^{I_{0}}$ in $\left\langle \omega,\eta\right\rangle ^{CW}$
is $\left\langle \omega_{I_{0}},\eta\right\rangle \neq0$ we find
that $\left\langle \cdot\right\rangle ^{CW}$ is non-degenerate. 
\end{proof}
For general $L$, the induced pairing on cohomology need not be perfect.
Therefore, we will need to introduce an additional assumption on $L$. 
\begin{defn}
\label{def:even cohomology}We say a closed manifold $L$ has \emph{even
cohomology }if 
\[
H^{2i+1,j}\left(C\left(L\right),d\right)=0
\]
for all $i,j$. 
\end{defn}
In particular we find that $L$ is both $\underline{\cc}$-separated
and $Or\left(TL\right)$-separated, see Definition \ref{def:separated},
so by Proposition \ref{prop:equivariant cohomology of separated manifolds}
we have an isomorphism of $\ca$-modules 
\[
\Phi:\cc\left[\vec{\alpha}\right]\otimes_{\cc}H\left(L\right)\to H^{CW}\left(L\right)
\]
(recall $H\left(L\right)=H^{\bullet,0}\left(L\right)\oplus H^{\bullet,1}\left(L\right)$
combines cohomology with values in $\underline{\cc}$ and in $Or\left(TL\right)$).
\begin{lem}
\label{lem:equivariant pairing normalized}We can choose the isomorphism
$\Phi$ so that 
\begin{equation}
\left\langle \Phi\left(f_{1}\otimes\psi_{1}\right),\Phi\left(f_{2}\otimes\psi_{2}\right)\right\rangle ^{CW}=f_{1}\cdot f_{2}\cdot\left\langle \psi_{1},\psi_{2}\right\rangle \label{eq:pairing preserved}
\end{equation}
for $\psi_{i}\in H\left(L\right)$ and $f_{i}\in\cc\left[\vec{\alpha}\right]$. 
\end{lem}
Hereafter, when there's no risk of confusion, we may use $\left\langle \cdot\right\rangle $
and $\left\langle \cdot\right\rangle ^{CW}$ to denote the induced
pairings on $H\left(L\right)$ and $H^{CW}\left(L\right)$, respectively.

\begin{proof}
Fix some basis $\left\{ \gamma_{i}\right\} _{i=0}^{2N-1}$ for $H\left(L\right)$
so that $\left\{ \gamma_{i}\right\} _{i=0}^{N-1}$ is a basis for
$H^{\bullet,0}\left(L\right)$ and $\left\{ \gamma_{i}\right\} _{i=N}^{2N-1}$
is the dual basis for $H^{\bullet,1}\left(L\right)$, i.e. we have
$\left\langle \gamma_{i},\gamma_{N+j}\right\rangle =\delta_{i,j}$
for ${0\leq i,j\leq N-1}$. It follows from $\cc\left[\vec{\alpha}\right]$-linearity
of $\left\langle \cdot\right\rangle ^{CW}$ that 
\[
C_{ij}:=\left\langle \Phi\left(1\otimes\gamma_{i}\right),\Phi\left(1\otimes\gamma_{N+j}\right)\right\rangle =\delta_{ij}\mod\mathfrak{m},
\]
so that $\left(C_{ij}\right)$ is an invertible $N\times N$ matrix
with coefficients in $\cc\left[\vec{\alpha}\right]$. Let $\left(C_{ij}'\right)$
denote its inverse, and let $ $$c:\ca\otimes H\left(L\right)\to\ca\otimes H\left(L\right)$
denote the $\ca$-linear automorphism defined in the basis $\gamma_{i}$
by the block-diagonal matrix $\begin{pmatrix}\mbox{id} & 0\\
0 & C_{ij}'
\end{pmatrix}$. The isomorphism $\Phi\circ c$ satisfies Eq (\ref{eq:pairing preserved}).\end{proof}
\begin{cor}
\label{cor:perfect equivariant pairing}If $L$ has even cohomology,
then the equivariant pairing $\left\langle \cdot\right\rangle ^{CW}$
induces a perfect pairing on $H^{CW}\left(L\right)$.
\end{cor}
We will see that the equivariant deformation of a cyclic invariant
algebra is cyclic. In $\S$\ref{sub:Cyclic equivariant retraction}
we will construct a unital cyclic retraction of $C^{CW}\left(L\right)$
to $H^{CW}\left(L\right)$. This means that we can apply Theorem \ref{thm:homological perturbation lemma}
to construct minimal models for equivariant extensions of invariant
twisted $A_{\infty}$ algebras.

\subsection{\label{sub:extension of invariant algebras}Equivariant extension
of invariant deformations}

Let $L$ be a non-orientable closed $\tb$-manifold, let $C=C\left(L\right)$
be the associated DGA over $R=\cc$ (cf. Example \ref{exa:De Rham DGA}).
Let $\bb=\bb_{\Lambda_{0}^{G}\left(\cc\right)}\left(C^{G}\right)$
be the associated bar coalgebra, defined over the Novikov ring $\Lambda_{0}^{G}\left(\cc\right)$.
For $1\leq a\leq n'$ define $\iota_{\xi_{a}}'$ by $\iota_{\xi_{a}}'x=\left(-1\right)^{\codim x}\iota_{\xi_{a}}$.
Let $\hat{\iota}_{\xi_{a}}$ denote the $\left(\mbox{id}_{\bb},\mbox{id}_{\bb}\right)$-coderivation
of degree $-1$ of $\bb$, corresponding to $\iota'_{\xi_{a}}$. Explicitly,
we have 
\[
\hat{\iota}_{\xi_{a}}\left(x_{1}\otimes\cdots\otimes x_{k}\right)=\sum_{i=1}^{k}\left(-1\right)^{\sum_{a=1}^{i-1}\left(\codim x_{a}-1\right)}x_{1}\otimes\cdots\otimes\iota_{\xi_{a}}'x_{i}\otimes\cdots\otimes x_{k}.
\]

\begin{defn}
\label{def:invariant twisted A8-1}$\left(C\left(L\right),\left\{ m_{k,\beta}\right\} \right)$
will be called $\mathbb{T}$-invariant if 
\begin{equation}
\iota_{\xi_{a}}'m_{k,\beta}\left(x_{1},...,x_{k}\right)=\sum_{i=1}^{k}\left(-1\right)^{1+\mu\left(\beta\right)+\sum_{j=1}^{i-1}\left(x_{j}-1\right)}m_{k,\beta}\left(x_{1},...,\iota_{\xi_{a}}'x_{i},...,x_{k}\right)\label{eq:m invariance-1}
\end{equation}
for all $\left(k,\beta\right)\neq\left(1,\mathbf{0}\right)$, $x_{1},...,x_{k}\in C\left(L\right)$
and $1\leq a\leq n'$. 
\end{defn}
For $j=1,2$ let $h^{j}:\bb\left[-d_{j}\right]\to\bb$ be maps in
$\hfgvect_{\cc}$, whose degree is ${d_{j}=\codim h^{j}}$. We define
the graded commutator $\left[h^{1},h^{2}\right]:=h^{1}\circ h^{2}-\left(-1\right)^{d_{1}\cdot d_{2}}h^{2}\circ h^{1}$.
\begin{lem}
\label{lem:bracket of coderivations}If $h^{1},h^{2}$ are $\left(\mbox{id}_{\bb},\mbox{id}_{\bb}\right)$-coderivations
then $\left[h^{1},h^{2}\right]$ is a degree $\left(\codim h^{1}+\codim h^{2}\right)$
$\left(\mbox{id}_{\bb},\mbox{id}_{\bb}\right)$-coderivation. In terms
of the bijection (\ref{eq:coderivations in components}), we have
\begin{multline*}
\left[h^{1},h^{2}\right]_{k,\beta}\left(x_{1}\otimes\cdots\otimes x_{k}\right)=\\
\sum\left(-1\right)^{\codim h_{2}\cdot\sum_{j<i}\left(\codim x_{j}-1\right)}h_{k_{1},\beta_{1}}^{1}\left(\cdots\otimes h_{k_{2},\beta_{2}}^{2}\left(x_{i}\otimes\cdots\right)\otimes\cdots\right)-\\
-\sum\left(-1\right)^{\codim h_{1}\cdot\left(\codim h_{2}+\sum_{j<i}\left(\codim x_{j}-1\right)\right)}h_{k_{2},\beta_{2}}^{2}\left(\cdots\otimes h_{k_{1},\beta_{1}}^{1}\left(x_{i}\otimes\cdots\right)\otimes\cdots\right)
\end{multline*}
\end{lem}
\begin{proof}
The first statement holds for any coalgebra, and is proved by dualizing
the proof of the corresponding statement for algebras. The second
claim is straightforward.
\end{proof}
The graded Jacobi identity 
\begin{equation}
\left(-1\right)^{\codim h_{1}h_{3}}\left[h_{1},\left[h_{2},h_{3}\right]\right]+\left(-1\right)^{\codim h_{1}\codim h_{2}}\left[h_{2},\left[h_{3},h_{1}\right]\right]+\left(-1\right)^{\codim h_{2}\codim h_{3}}\left[h_{3},\left[h_{1},h_{2}\right]\right]=0\label{eq:graded jacobi}
\end{equation}
holds, making the $\left(\mbox{id}_{\bb},\mbox{id}_{\bb}\right)$-coderivations
into a graded Lie algebra.

If we set $\partial=m-\hat{m}_{1,0}$ then we can write Eq (\ref{eq:m invariance-1})
as 
\[
\left[\hat{\iota}_{\xi_{a}},\partial\right]:=\hat{\iota}_{\xi_{a}}\partial+\partial\hat{\iota}_{\xi_{a}}=0\;\mbox{for }1\leq a\leq n',
\]
which implies that 
\[
\left[\hat{\iota}_{\xi_{a}},m\right]=\left[\hat{\iota}_{\xi_{a}},\hat{m}_{1,\mathbf{0}}\right]=\left(\iota_{\xi_{a}}'m_{1,\mathbf{0}}+m_{1,\mathbf{0}}\iota_{\xi_{a}}'\right)^{\wedge}=-\hat{\lc}_{\xi_{a}}
\]
where ${\hat{\lc}_{\xi_{a}}\left(x_{1}\otimes\cdots\otimes x_{k}\right)=\sum_{j=1}^{k}\left(x_{1}\otimes\cdots\otimes\lc_{\xi_{a}}x_{j}\otimes\cdots\otimes x_{k}\right)}$.
Note that we have $\left[\hat{\iota}_{\xi_{a}},\hat{\iota}_{\xi_{b}}\right]=0,\left[\hat{\iota}_{\xi_{a}},\hat{\lc}_{\xi_{b}}\right]=0$
and $\left[\hat{\lc}_{\xi_{a}},\hat{\lc}_{\xi_{b}}\right]=0$ for
all $1\leq a,b\leq n'$, by Lemma \ref{lem:bracket of coderivations}
and the commutativity of the vector fields $\left\{ \xi_{a}\right\} $.
\begin{prop}
\label{prop:equivariant twisted extension-1}Let $L$ be a non-orientable
closed $\tb$-manifold, and let $\left(C\left(L\right),d,\wedge\right)$
denote the associated $\grading$-graded De Rham DGA over $\cc$.
Let $\left(C\left(L\right),\left\{ m_{k,\beta}\right\} \right)$ be
a $G$-gapped $\mathbb{T}$-invariant deformation of $\left(C\left(L\right),d,\wedge\right)$.
We define operations 
\[
m_{k,\beta}^{CW}:C^{CW}\left(L\right)^{\otimes k}\to C^{CW}\left(L\right)\left[2-k-\mu\left(\beta\right),\mu\left(\beta\right)\mod2\right]
\]
as follows. For $\left(k,\beta\right)\neq\left(1,\mathbf{0}\right)$
we take $m_{k,\beta}^{CW}$ to be the $\cc\left[\vec{\alpha}\right]$-linear
extension of the restriction of $m_{k,\beta}$ to the $\mathbb{T}$-invariant
forms. For $\left(k,\beta\right)=\left(1,\mathbf{0}\right)$ we set
\emph{
\[
m_{1,\mathbf{0}}^{CW}x=\left(-1\right)^{\codim x}Dx.
\]
}Then 

(a) $\left(C^{CW}\left(L\right),\left\{ m_{k,\beta}^{CW}\right\} \right)$
is a $G$-gapped deformation of $\left(C^{CW}\left(L\right),D,\wedge\right)$
over $R=\cc\left[\vec{\alpha}\right]$.

(b) If $\left(C\left(L\right),\left\{ m_{k,\beta}\right\} \right)$
is unital, then $\left(C^{CW}\left(L\right),\left\{ m_{k,\beta}^{CW}\right\} \right)$
is a unital deformation of $\left(C^{CW}\left(L\right),D,\wedge\right)$.

\noindent Assuming, in addition, that $L$ has even cohomology, we
have

(c) If $\left(C\left(L\right),\left\{ m_{k,\beta}\right\} ,\left\langle \cdot\right\rangle \right)$
is cyclic then $\left(C^{CW}\left(L\right),\left\{ m_{k,\beta}^{CW}\right\} ,\left\langle \cdot\right\rangle ^{CW}\right)$
is a cyclic deformation of $\left(C^{CW}\left(L\right),D,\wedge,\int^{CW}\right)$.\end{prop}
\begin{proof}
To prove (a), we first show that the operations $\left\{ m_{k,\beta}^{CW}\right\} $
are well-defined (that is, take invariant forms to invariant forms)
and satisfy the twisted $A_{\infty}$ relations (\ref{eq:twisted A8 relations in components}).
One can check this directly, but the signs get somewhat cumbersome.
It is convenient to switch to the bar coalgebra, using the setup introduced
just before the proposition. Consider $\bb':=\bb\left(\left(C\left(L\right)\otimes\ca\right)^{G}\right)=\bb\hotimes_{\cc}\cc\left[\vec{\alpha}\right]$
(as always, completion is with respect to the energy filtration, with
$\cc\left[\vec{\alpha}\right]$ discrete). We'll use $\hat{\iota}_{\xi_{a}},\hat{\lc}_{\xi_{b}}$
and $m$ to denote, by the usual abuse of notation, the extension
of scalars of these coderivations from $\bb$ to $\bb'$. Consider
the $\left(\mbox{id}_{\bb'},\mbox{id}_{\bb'}\right)$-coderivation
\[
m':=m-\sum_{a}\alpha_{a}\hat{\iota}_{\xi_{a}}.
\]

We have

\begin{equation}
\left[m',m'\right]=-2\sum_{a}\left[\alpha_{a}\hat{\iota}_{\xi_{a}},m\right]=-2\sum_{a}\alpha_{a}\hat{\lc}_{\xi_{a}}.\label{eq:equiv m squares to lie derivative}
\end{equation}
By the graded Jacobi identity,
\begin{multline*}
0=\left[\hat{\iota}_{\xi_{a}},\left[m',m'\right]\right]+2\left[m',\left[\hat{\iota}_{\xi_{a}},m'\right]\right]=\\
\left[\hat{\iota}_{\xi_{a}},-2\sum_{b}\alpha_{b}\hat{\lc}_{\xi_{b}}\right]+2\left[m',\left[\hat{\iota}_{\xi_{a}},m-\sum_{b}\alpha_{b}\hat{\iota}_{\xi_{b}}\right]\right]=-2\left[m',\hat{\lc}_{\xi_{a}}\right],
\end{multline*}
which, by Lemma \ref{lem:bracket of coderivations}, implies 
\begin{equation}
\lc_{\xi_{a}}m_{k,\beta}\left(x_{1}\otimes\cdots\otimes x_{k}\right)=\sum_{i}m_{k,\beta}\left(\cdots\otimes\lc_{\xi_{a}}x_{i}\otimes\cdots\right),\;\left(k,\beta\right)\neq\left(1,\mathbf{0}\right).\label{eq:m commutes with lie}
\end{equation}
Eq (\ref{eq:m commutes with lie}) shows that $m'$ preserves the
sub-coalgebra 
\[
{\bb^{CW}:=\bb_{\Lambda_{0}^{G}\left(\ca\right)}\left(\left(C^{CW}\left(L\right)\right)^{G}\right)\subset\bb'}
\]
($\bb^{CW}$ is the completion of the $\Lambda_{0}^{G}\left(\ca\right)$-span
of $\left\{ x_{1}\otimes\cdots\otimes x_{k}\in\bb'|\lc_{\xi_{a}}x_{i}=0,\;1\leq i\leq k,1\leq a\leq n'\right\} $).

By Eq (\ref{eq:equiv m squares to lie derivative}) $m^{CW}:=m'|_{\bb^{CW}}$
is a differential, whose components $\left\{ m_{k,\beta}^{CW}\right\} $
are as prescribed in the statement of the proposition. Clearly $m_{0,0}^{CW}=0$
so $\left\{ m_{k,\beta}^{CW}\right\} $ define a twisted $A_{\infty}$
algebra, which is easily seen to be a deformation of $\left(C^{CW}\left(L\right),D,\wedge\right)$.

To establish (b) we note that if $m_{k,\beta}$ is unital the unit
\textbf{$\mathbf{e}$} must be the constant 0-form 1. Since $\lc_{\xi_{r}}\mathbf{e}=0$
and $\iota_{\xi_{r}}\mathbf{e}=0$ for all $1\leq r\leq n'$, we find
that $\mathbf{e}$ is also a unit for $\left\{ m_{k,\beta}^{CW}\right\} $. 

Finally we prove (c). First, we check that Eq (\ref{eq:twisted cyclic symmetry})
holds for $\left\{ m_{k,\beta}^{CW}\right\} $. For ${\left(k,\beta\right)\neq\left(1,\mathbf{0}\right)}$
this is immediate, and for $\left(k,\beta\right)=\left(1,\mathbf{0}\right)$
it is a special case of Example \ref{exa:DGA}. Then (c) follows from
Lemma \ref{lem:CW pairing is non-degenerate} and Corollary \ref{cor:perfect equivariant pairing}.\end{proof}
\begin{rem}
\label{rem:T star modules}A \emph{$\tb^{*}$ module} is an algebraic
structure which captures the notion of a smooth action of a Lie group
on a differential complex, see \cite{supersymmetry+equivariance}.
This structure suffices to define an equivariant extension of the
complex, in essentially the same way as for the De Rham complex of
a manifold $L$. We would like to point out that Definition \ref{def:invariant twisted A8-1}
and the proof of Proposition \ref{prop:equivariant twisted extension-1}
can be formulated in this more general language. In other words, instead
of starting off from $C=C\left(L\right)$ and working with deformations
of the De Rham DGA, we can assume that $\left(C,d\right)$ admits
a $\tb^{*}$-module structure, define when perturbations\emph{ }of
$d$ are $\tb$-invariant, and for such perturbations obtain equivariant
extensions as above.
\end{rem}

\subsection{\label{sub:Cyclic equivariant retraction}Cyclic unital equivariant
retraction}

The main goal of this section is to prove the following.
\begin{thm}
\label{thm:equivariant cyclic retraction exists}Let $L$ be a non-orientable
closed $\mathbb{T}$-manifold which has even cohomology. Then there
exists a cyclic unital retraction of $\left(C^{CW}\left(L\right),D\right)$
to $H\left(C^{CW}\right)$. \end{thm}
\begin{rem}
All the results in this section continue to hold for $L$ oriented,
\emph{mutatis mutandis.}
\end{rem}
In proving this theorem we will construct an \emph{equivariant homotopy
kernel, }which represents the homotopy operator of the retraction,
and plays a prominent role in defining equivariant open Gromov-Witten
invariants. This result can be seen as a kind of equivariant extension
of the Hodge-De Rham decomposition.

Let $L$ be a non-orientable $n$-dimensional closed $\tb$-manifold.
For $0\leq i\leq2N-1$ choose $\omega_{i}\in C^{CW}\left(L\right)$
such that (i) $D\omega_{i}=0$, (ii) the cohomology classes $\left\{ [\omega_{i}]\right\} _{i=0}^{2N-1}$
form a basis for $H^{CW}\left(L\right)$, (iii) $\omega_{i}\in\left(C^{CW}\left(L\right)\right)^{\bullet,0}$
for $0\leq i\leq N-1$ and $\omega_{i}\in\left(C^{CW}\left(L\right)\right)^{\bullet,1}$
for $N\leq i\leq2N-1$, and (iv) $\left\langle \omega_{i},\omega_{j+N}\right\rangle =\delta_{i,j}$.
A set of forms $\left\{ \omega_{i}\right\} _{i=0}^{2N-1}$ satisfying
(i) - (iv) will be called an \emph{equivariant basis of forms for
$L$.} An equivariant basis of forms exists; indeed, one can take
$\omega_{i}$ to be any form which represents $\Phi\left(1\otimes\gamma_{i}\right)$,
see the proof of Lemma \ref{lem:equivariant pairing normalized} and
the discussion directly above it for explanation of the notation.

Consider the diagonal $L=\Delta\overset{i_{\Delta}}{\hookrightarrow}L\times L$,
and let $i_{\Delta}\left(L\right)\subset N_{\Delta}\subset L\times L$
be an equivariant tubular neighbourhood. Construct the blow up $\beta_{\Delta}:\widetilde{L\times L}\to L\times L$,
where $\widetilde{L\times L}=\bl_{\Delta}\left(L\times L\right)=\widetilde{N}_{\Delta}\bigcup\left(L\times L\backslash\Delta\right)$,
see $\S$\ref{sub:Blowing-up}. The blow up $\widetilde{L\times L}$
is a compact manifold with boundary $i_{\partial}:S\left(N_{\Delta}\right)\hookrightarrow\widetilde{L\times L}$.
For $j=1,2$ we denote by $pr_{j}:L\times L\to L$ the projection
and set $\widetilde{pr}_{j}:=pr_{j}\circ\beta_{\Delta}$.

Note that 
\[
0\to TL\to i_{\Delta}^{*}\left(pr_{1}^{*}TL\oplus pr_{2}^{*}TL\right)\to N_{\Delta}\to0.
\]
Since we have $pr_{j}\circ i_{\Delta}=\mbox{id}_{\Delta}$ and $i_{\Delta}\pi_{S}=\tilde{\pi}i_{\partial}$
we find that $Or\left(N_{\Delta}\right)\simeq i_{\Delta}^{*}Or\left(pr_{2}^{*}TL\right)$
and 
\begin{equation}
\pi_{S}^{*}Or\left(N_{\Delta}\right)^{\vee}\simeq i_{\partial}^{*}\widetilde{pr}_{2}^{*}Or\left(TL\right)^{\vee}\label{eq:normal to diagonal orientation}
\end{equation}

\begin{defn}
\label{def:homotopy kernel}Let $L$ be a non-orientable closed $\tb$-manifold
with even cohomology. An \emph{equivariant homotopy kernel $\Lambda'$
for }$L$ is a form $\Lambda'\in\Omega\left(\widetilde{L\times L};\widetilde{pr}_{2}^{*}\left(Or\left(TL\right)^{\vee}\right)\right)$
with the following properties.

(a) $i_{S}^{*}\Lambda'=-\phi+i_{S}^{*}\beta_{\Delta}^{*}\Upsilon$
where $\phi\in\Omega\left(S\left(N\right);i_{\partial}^{*}\widetilde{pr}_{2}^{*}Or\left(TL\right)^{\vee}\left[\vec{\alpha}\right]\right)^{\tb}$
is an equivariant angular form for $S\left(N_{\Delta}\right)$ (see
(\ref{eq:normal to diagonal orientation})) and $\Upsilon\in\Omega\left(L\times L;pr_{2}^{*}Or\left(TL\right)^{\vee}\left[\vec{\alpha}\right]\right)^{\tb}$. 

(b) $D\Lambda'=\sum_{j=0}^{N-1}\widetilde{pr}{}_{1}^{*}\omega_{j}\wedge\widetilde{pr}{}_{2}^{*}\omega_{j+N}$
where $\left\{ \omega_{i}\right\} _{i=0}^{2N-1}$ is an equivariant
basis of forms for $L$. \end{defn}
\begin{prop}
\label{prop:homotopy kernel exists}If $L$ is a non-orientable closed
$\tb$-manifold which has even cohomology, an equivariant homotopy
kernel for $L$ exists.\end{prop}
\begin{proof}
Denote $\lc=pr_{2}^{*}Or\left(TL\right)^{\vee}$. Use (\ref{eq:normal to diagonal orientation})
to define the isomorphism of local systems (\ref{eq:ls iso needed})
\[
\kappa:Or\left(TL\right)\otimes i_{\Delta}^{*}Or\left(T\left(L\times L\right)\right)^{\vee}\simeq Or\left(N_{\Delta}^{\vee}\right)\simeq i_{\Delta}^{*}\lc
\]
needed in order to construct an equivariant Poincare dual to the diagonal
$\tau_{\Delta}$ with values in $pr_{2}^{*}Or\left(TL\right)^{\vee}$
(see \ref{cor:equivariant poincare duality}). Let $\check{\Lambda}\in\Omega^{n-1}\left(\widetilde{L\times L};\widetilde{pr}{}_{2}^{*}Or\left(TL\right)\left[\vec{\alpha}\right]\right)^{\tb}$
be the unique form supported in $\widetilde{N}_{\Delta}\subset\bl_{\Delta}\left(L\times L\right)$
such that 
\[
\check{\Lambda}|_{\widetilde{N}_{\Delta}}=pr_{S}^{*}\phi\wedge pr_{r}^{*}\sigma,
\]
so that $D\check{\Lambda}=\beta_{\Delta}^{*}\tau_{\Delta}$ (cf. Eq
(\ref{eq:thom defining equation})). 

Let $\left\{ \omega_{i}\right\} _{i=0}^{2N-1}$ be an equivariant
basis for $L\times L$. By Corollary \ref{cor:equivariant Kunneth}
we have some elements $f_{ij}\in\cc\left[\vec{\alpha}\right]$ with
\begin{equation}
\left[\tau\right]=\sum_{0\leq i,j\leq2N-1}f_{ij}\cdot\left(pr_{1}^{*}\left[\omega_{i}\right]\wedge pr_{2}^{*}\left[\omega_{j}\right]\right)\label{eq:thom decomposed by Kunneth}
\end{equation}
We claim 
\[
f_{ij}=\delta_{i,j+N}:=\begin{cases}
1 & i+j=N\mod2N\\
0 & \mbox{otherwise}
\end{cases}.
\]
To see this, we compute $\int_{L\times L}\tau\wedge pr_{1}^{*}\omega_{i'}\wedge pr_{2}^{*}\omega_{j'}$
for $0\leq i',j'\leq2N-1$, in two ways. First we use Corollary \ref{cor:equivariant poincare duality}:
\begin{equation}
\int_{L\times L}\tau\wedge pr_{1}^{*}\omega_{i'}\wedge pr_{2}^{*}\omega_{j'}=\int_{L}\omega_{i'}\wedge\omega_{j'}=\delta_{i',j'+N}.\label{eq:pair thom using PD}
\end{equation}
On the other hand, by Eq (\ref{eq:thom decomposed by Kunneth}) we
have:

\begin{multline}
\int_{L\times L}\tau\wedge pr_{1}^{*}\omega_{i'}\wedge pr_{2}^{*}\omega_{j'}=\sum_{i,j}f_{ij}\int_{L\times L}pr_{1}^{*}\left(\omega_{i}\wedge\omega_{i'}\right)\wedge pr_{2}^{*}\left(\omega_{j}\wedge\omega_{j'}\right)=\\
=\sum_{i,j}f_{ij}\left(\int_{L}\omega_{i}\wedge\omega_{i'}\right)\left(\int_{L}\omega_{j}\wedge\omega_{j'}\right)=\sum_{i,j}f_{ij}\delta_{i,i'+N}\delta_{j,j'+N}=f_{i'+N,j'+N}.\label{eq:pair thom directly}
\end{multline}
Combining Eqs (\ref{eq:pair thom using PD},\ref{eq:pair thom directly})
we conclude $f_{ij}=\delta_{i,j+N}$.

Now define $\Lambda'=\check{\Lambda}+\beta_{\Delta}^{*}\Upsilon$
where $D\Upsilon=\sum_{i=0}^{N-1}\widetilde{pr}{}_{1}^{*}\omega_{i}\wedge\widetilde{pr}{}_{2}^{*}\omega_{i+N}-\tau$
so 
\[
D\Lambda'=\sum_{i=0}^{N-1}\widetilde{pr}{}_{1}^{*}\omega_{i}\wedge\widetilde{pr}{}_{2}^{*}\omega_{i+N}
\]
and we have $i_{S}^{*}\Lambda'=-\phi+i_{S}^{*}\beta_{\Delta}^{*}\Upsilon$
by construction.
\end{proof}

Given an equivariant homotopy kernel $\Lambda'$ for $L$ we define
$\ca$-module maps ${\Pi_{1,0}:C^{CW}\left(L\right)\to H^{CW}\left(L\right)}$,
${I_{1,0}:H^{CW}\left(L\right)\to C^{CW}\left(L\right)}$, and ${h_{1,0}':C^{CW}\left(L\right)[+1]\to C^{CW}\left(L\right)}$,
by setting

\begin{gather}
h'_{1,\mathbf{0}}x=\begin{cases}
\widetilde{pr}_{1!}\left(\widetilde{pr}_{2}^{*}x\;\Lambda'\right) & \ls x=0\\
\widetilde{pr}_{2!}\left(\widetilde{pr}_{1}^{*}x\;\Lambda'\right) & \ls x=1
\end{cases}.\nonumber \\
I_{1,\mathbf{0}}\left[\omega_{i}\right]=\omega_{i},\qquad0\leq i\leq2N-1\label{eq:almost retraction from kernel}\\
h'_{1,\mathbf{0}}x=\begin{cases}
\widetilde{pr}_{1!}\left(\widetilde{pr}_{2}^{*}x\;\Lambda'\right) & \ls x=0\\
\widetilde{pr}_{2!}\left(\widetilde{pr}_{1}^{*}x\;\Lambda'\right) & \ls x=1
\end{cases}.\nonumber 
\end{gather}

\[
\]

Let us explain how the pushforward is defined. We deal with the case
$\ls x=0$, the other case is analogous. It is easy to see that $\widetilde{pr}_{1}$
is an equivariant proper submersion, so by Definition \ref{def:oriented submersion}
we just need to supply an equivariant local system isomorphism

\[
Or\left(T\left(\widetilde{L\times L}\right)\right)\otimes\widetilde{pr}{}_{1}^{*}Or\left(TL\right)^{\vee}\to\kk^{\vee}\otimes\widetilde{pr}_{1}^{*}\lc,
\]
where we have $\kk=\widetilde{pr}_{2}^{*}Or\left(TL\right)^{\vee}$
and require $\lc=\underline{\cc}$ in this case (recall we need $h'_{1,\mathbf{0}}$
to preserve the local system degree). By part (b) of Lemma \ref{lem:blow up is relatively orientable}
we have 
\[
Or\left(T\left(\widetilde{L\times L}\right)\right)\otimes\widetilde{pr}{}_{1}^{*}Or\left(TL\right)^{\vee}\simeq\beta_{\Delta}^{*}\left(Or\left(T\left(L\times L\right)\right)\otimes pr_{1}^{*}Or\left(TL\right)^{\vee}\right),
\]
and we compose this with $\beta_{\Delta}^{*}\kappa'$ where $\kappa'$
is the composition of the obvious local system isomorphisms
\[
Or\left(T\left(L\times L\right)\right)\otimes pr_{1}^{*}Or\left(TL\right)^{\vee}\simeq pr_{2}^{*}Or\left(TL\right)^{\vee}\to\kk'\otimes pr_{1}^{*}\lc
\]
for $\kk'=pr_{2}^{*}Or\left(TL\right)^{\vee}$.

\begin{lem}
\label{lem:kernel gives almost retraction}$\left(\Pi_{1,0},I_{1,0},h'_{1,0}\right)$
satisfies all the conditions of Definition \ref{def:retraction} for
being a retraction of $\left(C^{CW}\left(L\right),D\right)$ to $H^{CW}\left(L\right)$,
\emph{except} possibly for the side conditions (\ref{eq:side conditions}).
Moreover, it satisfies Eq (\ref{eq:cyclic retraction}) and $h_{1,\mathbf{0}}'1^{0,0}=0$
where $1^{0,0}\in\left(C^{CW}\left(L\right)\right)^{0,0}$ denotes
the constantly 1 zero-form. \end{lem}
\begin{proof}
$\Pi_{1,\mathbf{0}}D=0$ by Stokes' theorem. $DI_{1,\mathbf{0}}=0$
because $\omega_{i}$ are closed. $\Pi_{1,\mathbf{0}}I_{1,\mathbf{0}}=\mbox{id}_{H^{CW}\left(L\right)}$
since this is so for the basis elements by the assumption on the pairing. 

Let us check that $D'h_{1,\mathbf{0}}+h_{1,\mathbf{0}}D'=I_{1,\mathbf{0}}\Pi_{1,\mathbf{0}}-\mbox{id}_{C^{CW}\left(L\right)}$.
We assume $\ls x=0$, the case $\ls x=1$ is analogous.

\begin{multline}
\left(D'h_{1,\mathbf{0}}+h_{1,\mathbf{0}}D'\right)x=\left(-1\right)^{\codim x}\left[-D\widetilde{pr}_{1!}\left(\widetilde{pr}_{2}^{*}x\Lambda'\right)+\widetilde{pr}_{1!}\left(\widetilde{pr}_{2}^{*}Dx\Lambda'\right)\right]=\\
=\left(-1\right)^{\codim x}\left[-D\widetilde{pr}_{1!}\left(\widetilde{pr}_{2}^{*}x\Lambda'\right)+\widetilde{pr}_{1!}\left(D\left(\widetilde{pr}_{2}^{*}x\wedge\Lambda'\right)\right)+\left(-1\right)^{\codim x}\widetilde{pr}_{1!}\left(\widetilde{pr}_{2}^{*}x\wedge D\left(\Lambda'\right)\right)\right]\\
=\left(-1\right)^{\codim x}\left[-D\widetilde{pr}_{1!}\left(\widetilde{pr}_{2}^{*}x\;\Lambda'\right)+\widetilde{pr}_{1!}\left(D\left(\widetilde{pr}_{2}^{*}x\;\Lambda'\right)\right)\right]+\\
+\widetilde{pr}_{1!}\left(\widetilde{pr}_{2}^{*}x\wedge\sum_{i=0}^{N-1}\widetilde{pr}{}_{1}^{*}\omega_{i}\wedge\widetilde{pr}{}_{2}^{*}\omega_{i+N}\right).\label{eq:commutator w D all}
\end{multline}

We show

\begin{equation}
\left(-1\right)^{\codim x}\left[-D\widetilde{pr}_{1!}\left(\widetilde{pr}_{2}^{*}x\;\Lambda'\right)+\widetilde{pr}_{1!}\left(D\left(\widetilde{pr}_{2}^{*}x\;\Lambda'\right)\right)\right]=-\mbox{id}x.\label{eq:commutator w D left}
\end{equation}
Indeed,

\begin{multline*}
\left(-1\right)^{\codim x}\left[-D\widetilde{pr}_{1!}\left(\widetilde{pr}_{2}^{*}x\;\Lambda'\right)+\widetilde{pr}_{1!}\left(D\left(\widetilde{pr}_{2}^{*}x\;\Lambda'\right)\right)\right]=\\
=\left(-1\right)^{\codim x}\left(-1\right)^{\codim x+\left(n-1\right)+n}\left(-1\right)\pi_{S!}\left(i_{S}^{*}\widetilde{pr}_{2}^{*}x\;\left(-\phi\right)+\pi_{S}^{*}i_{\Delta}^{*}\epsilon\right)
\end{multline*}
The factor $\left(-1\right)^{\codim x+\left(n-1\right)+n}$ comes
from Stokes' theorem \ref{lem:Stokes'-Theorem.}. The $\left(-1\right)$
immediately following it appears because we change from the orientation
of $S\left(N_{\Delta}\right)$ as a boundary of $\widetilde{L\times L}$
to the orientation as the boundary of the unit $D\left(N_{\Delta}\right)$.
We have $i_{S}^{*}\widetilde{pr}_{2}^{*}x=\pi_{S}^{*}x$ and $\pi_{S!}\phi=1$,
so by the projection formula \ref{lem:(Projection-Formula.)}

\[
\left(-1\right)^{\codim x}\left(-1\right)^{\codim x+\left(n-1\right)+n}\left(-1\right)\pi_{S!}\left(i_{S}^{*}\widetilde{pr}_{2}^{*}x\;\left(-\phi\right)+\pi_{S}^{*}i_{\Delta}^{*}\epsilon\right)=\left(-\mbox{id}\right)x,
\]
as claimed. 

Next we show
\begin{equation}
\widetilde{pr}_{1!}\left(\widetilde{pr}_{2}^{*}x\wedge\sum_{i=0}^{N-1}\widetilde{pr}{}_{1}^{*}\omega_{i}\wedge\widetilde{pr}{}_{2}^{*}\omega_{i+N}\right)=I_{1,\mathbf{0}}\Pi_{1,\mathbf{0}}x.\label{eq:commutator w D right}
\end{equation}
 First, we claim $\widetilde{pr}_{1!}^{\kappa}\beta^{*}=pr_{1!}^{\kappa'}$.
To see this, remove an arbitrarily small neighbourhood of $S\left(N_{\Delta}\right)\times\left\{ 0\right\} \subset\widetilde{L\times L}$,
so that $\beta$ can be taken to be the identity, and then $\kappa=\kappa'$
(see the second statement in Lemma \ref{lem:blow up is relatively orientable}
(b)). With this established, we compute:

\begin{multline*}
\widetilde{pr}_{1!}\left(\widetilde{pr}_{2}^{*}x\wedge\sum_{i=0}^{N-1}\widetilde{pr}{}_{1}^{*}\omega_{i}\wedge\widetilde{pr}{}_{2}^{*}\omega_{i+N}\right)=pr_{1!}\left(pr_{2}^{*}x\wedge\sum_{i=0}^{N-1}pr{}_{1}^{*}\omega_{i}\wedge pr{}_{2}^{*}\omega_{i+N}\right)=\\
=\sum_{i=0}^{N-1}pr_{1!}\left(pr_{2}^{*}\left(x\wedge\omega_{i+N}\right)\wedge pr{}_{1}^{*}\omega_{i}\right)=\sum_{i=0}^{N-1}\int_{L}\left(x\wedge\omega_{i+N}\right)\omega_{i}=I_{1,\mathbf{0}}\Pi_{1,\mathbf{0}}x.
\end{multline*}
plugging in Eqs (\ref{eq:commutator w D left},\ref{eq:commutator w D right})
on the last line of Eq (\ref{eq:commutator w D all}) shows that $Dh_{1,\mathbf{0}}+h_{1,\mathbf{0}}D=I_{1,\mathbf{0}}\Pi_{1,\mathbf{0}}-\mbox{id}_{C^{CW}\left(L\right)}$.

Next we check that Eq (\ref{eq:cyclic retraction}) holds. We assume
without loss of generality that $\ls x=0$ and $\ls y=1$; the other
cases either vanish since the integrand takes values in the trivial
local system, or can be derived from this case by the antisymmetry
of the pairing (\ref{eq:bracket antisymmetry}). We compute, using
(\ref{eq:pairing from functional}), Lemma \ref{lem:(Projection-Formula.)}
and the fact $\deg\widetilde{pr}_{2!}$ is even and $\deg\Lambda'$
is odd:

\begin{multline*}
\left(-1\right)^{\left(\codim x-1\right)\codim y+\left(\codim x-1\right)}\left\langle h_{1,0}x,y\right\rangle ^{CW}=\int_{L}\widetilde{pr}_{1!}\left(\widetilde{pr}_{2}^{*}x\;\Lambda'\right)\wedge y=\\
=\int_{\widetilde{L\times L}}\widetilde{pr}_{2}^{*}x\;\Lambda'\;\widetilde{pr}_{1}^{*}y=\int_{L}x\;\widetilde{pr}{}_{2!}\left(\Lambda'\;\widetilde{pr}_{1}^{*}y\right)=\\
=\left(-1\right)^{\codim y}\int_{L}x\;\widetilde{pr}{}_{2!}\left(\widetilde{pr}_{1}^{*}y\;\Lambda'\right)=\left(-1\right)^{\codim y}\left(-1\right)^{\codim x\left(\codim y-1\right)+\codim x}\left\langle x,h_{1,\mathbf{0}}y\right\rangle ^{CW}
\end{multline*}
which is equivalent to $\left\langle h_{1,0}x,y\right\rangle +\left(-1\right)^{\codim x}\left\langle x,h_{1,0}y\right\rangle =0$,
as we wanted to show.

Finally, $h_{1,\mathbf{0}}1^{0,0}=0$ since the degree of $\Lambda'$
is $n-1$ and the dimension of the fiber of $\widetilde{pr}_{1}$
is $n$.
\end{proof}
The following is a simple variation of a result found in \cite{side-cond-trick-encyc},
where it is attributed to Lambe and Stasheff \cite{side-cond-trick}.
\begin{lem}
\label{lem:correcting an almost retraction}Let $\left(C,d\right)$
be a dg $R$-module. Let $\Pi_{1,\mathbf{0}}:C\to HC$, $I_{1,\mathbf{0}}:HC\to C$,
and $h_{1,0}':C\to C\left[-1\right]$ be $R$-module maps such that
the 3-tuple $\left(\Pi_{1,0},I_{1,0},h'_{1,0}\right)$ satisfies all
the conditions of Definition \ref{def:retraction}, except perhaps
for the side conditions (\ref{eq:side conditions}).

Set 
\[
h_{1,0}=h_{1,0}'\circ d\circ h_{1,0}'.
\]
Then $\left(\Pi_{1,0},I_{1,0},h{}_{1,0}\right)$ is a retraction of
$\left(C,d\right)$ onto $H\left(C,d\right)$. If Eq (\ref{eq:cyclic retraction})
holds for $h'_{1,\mathbf{0}}$ then the retraction $\left(\Pi_{1,0},I_{1,0},h{}_{1,0}\right)$
is cyclic. If $h_{1,0}'\mathbf{e}=0$ then $\mathbf{e}$ is a unit
for $\left(\Pi_{1,0},I_{1,0},h_{1,0}\right)$. \end{lem}
\begin{proof}
straightforward. 
\end{proof}

\begin{proof}
[Proof of Theorem \ref{thm:equivariant cyclic retraction exists}]By
Proposition \ref{prop:homotopy kernel exists} an equivariant homotopy
kernel $\Lambda'$ exists. We define $\left(\Pi_{1,\mathbf{0}},I_{1,\mathbf{0}},h_{1,\mathbf{0}}'\right)$
by (\ref{eq:almost retraction from kernel}), and then apply Lemma
\ref{lem:kernel gives almost retraction} and Lemma \ref{lem:correcting an almost retraction}
to produce the desired cyclic unital retraction $\left(\Pi_{1,\mathbf{0}},I_{1,\mathbf{0}},h_{1,\mathbf{0}}\right)$
of $C^{CW}\left(L\right)$ to $H^{CW}\left(L\right)$.\end{proof}
\begin{rem}
It is not hard to show that $h_{1,\mathbf{0}}$ is also represented
by a kernel. That is, there exists some $\Lambda\in\Omega\left(\widetilde{L\times L};\widetilde{pr}_{2}^{*}\left(Or\left(TL\right)^{\vee}\right)\bva\right)$
such that 
\[
h{}_{1,\mathbf{0}}x=\begin{cases}
\widetilde{pr}_{1!}\left(\widetilde{pr}_{2}^{*}x\;\Lambda\right) & \ls x=0\\
\widetilde{pr}_{2!}\left(\widetilde{pr}_{1}^{*}x\;\Lambda\right) & \ls x=1
\end{cases}.
\]

\end{rem}

\section{\label{sec:Appendix.-Orientation-convention}Appendix. Orientation
conventions}

In this appendix we state some conventions regarding orientations
and the definition of the pushforward of forms; we then specify the
signs appearing in formulas involving the pushforward operation. We
work in the non-equivariant context, but the conventions and results
hold, \emph{mutatis mutandis}, for the equivariant setting.

\subsection{Terminology and pullback orientation.}

It will be convenient to speak in terms of relative orientations\emph{.
}A bundle $E\to B$ is called \emph{relatively oriented, }or \emph{r-oriented
}for short, if there is some specific isomorphism of local systems
between $Or\left(E\right)$ and a local system on $B$ which is treated
as known. The collection of known local systems is usually specified
by \emph{presupposing }some bundles to be r-oriented\emph{.} When
we say a manifold with boundary $X$ is r-oriented we mean $TX$ is
r-oriented.

As an example we state the following convention.

\medskip{}

\textbf{Convention. }If $E\to B$ is r-oriented and $f:B'\to B$ is
any smooth map, then the pullback bundle $f^{*}E\to B'$ is r-oriented
using the canonical isomorphism of local systems $Or\left(f^{*}E\right)\simeq f^{*}Or\left(E\right)$.

\subsection{\label{sub:2 of 3 orientation}Short exact sequences and ordered
direct sums}

In the smooth category, a short exact sequence of bundles over a base
space $B$ 
\[
0\to E_{1}\to E\to E_{2}\to0
\]
can always be split
\begin{equation}
E\simeq E_{1}\oplus E_{2}.\label{eq:ordered direct sum}
\end{equation}
For the purposes of orientation, it is important to remember the order
of the summands; to emphasize this we will sometimes call (\ref{eq:ordered direct sum})
an \emph{ordered }direct sum decomposition. A short exact sequence
always fixes the order of the summands as above.

If the rank of $E_{i}$ is $n_{i}$, and the rank of $E$ is $n=n_{1}+n_{2}$,
the wedge product ${\wedge:\Omega_{cv}^{n_{1}}\left(E_{1}\right)\otimes\Omega_{cv}^{n_{2}}\left(E_{2}\right)\to\Omega_{cv}^{n}\left(E_{1}\oplus E_{2}\right)}$
induces an isomorphism ${H_{cv}^{n_{1}}\left(E_{1}|_{U}\right)\otimes H_{cv}^{n_{2}}\left(E_{2}|_{U}\right)\simeq H_{cv}^{n}\left(E|_{U}\right)}$
for sufficiently small $U\subset B$, and thus an isomorphism 
\begin{equation}
Or\left(E\right)\simeq Or\left(E_{1}\right)\otimes Or\left(E_{2}\right).\label{eq:direct sum ls iso}
\end{equation}

\textbf{Convention.} If any two of the three vector bundles in a short
exact sequence, or an ordered direct sum decomposition, are r-oriented\emph{
}then the third is r-oriented by Eq (\ref{eq:direct sum ls iso}).

\subsection{\label{sub:boundary orientation}Boundary.}

If $X$ is a manifold with boundary we have an ordered direct sum
decomposition
\begin{equation}
TX|_{\partial X}=\underline{\rr}\oplus T\partial X\label{eq:boundary decomposition}
\end{equation}
which is \emph{the reverse} of the ordered direct sum decomposition
for the short exact sequence 
\[
0\to T\partial X\to TX|_{\partial X}\overset{-dc}{\longrightarrow}\underline{\rr}\to0
\]
where $c:X\to[0,1)$ is any collar coordinate. Henceforth we use $\underline{\rr}$
to denote the trivial real line bundle, if the base (in this case,
$\partial X$) is clear from the context. $\underline{\rr}$ is canonically
oriented, hence r-oriented, and we have the following.

\medskip{}

\textbf{Convention. }An r-orientation on $X$ induces an r-orientation
on $\partial X$ by Eq (\ref{eq:boundary decomposition}).

\medskip{}

In other words, we orient $\partial X$ so that an oriented basis
for $TX$ is obtained from an oriented basis for $T\partial X$ by
appending an \emph{outward} normal vector \emph{at the beginning}.

\subsection{\label{sub:Pushforward convention}Pushforward.}
\begin{defn}
\label{def:oriented submersion}A pair $\left(f,\kappa\right)$ will
be called \emph{an oriented proper submersion} if 
\begin{enumerate}
\item $f:X\to Y$ is a surjective, proper map between manifolds without
boundary, such that $df|_{x}$ is surjective for all $x\in X$.
\item $\kappa:Or\left(TX\right)\otimes f^{*}Or\left(TY\right)^{\vee}\simeq\kk^{\vee}\otimes f^{*}\lc$
is an isomorphism of local systems for $\kk$ (resp. $\lc$) some
local system on $X$ (resp. $Y$).
\end{enumerate}
\end{defn}

\begin{defn}
We say $\left(f,\kappa\right)$ is an \emph{oriented proper submersion
with boundary }if the pair $\left(f,\kappa\right)$ is as in the previous
definition, except $X$ now has (possibly empty) boundary (and $Y$
has no boundary) and we require in addition that $f|_{\partial X}$
be surjective and $d\left(f|_{\partial X}\right)|_{x}$ be surjective
for every $x\in\partial X$.
\end{defn}

\medskip{}
\textbf{Convention. }Let $\left(f,\kappa\right)$ be an oriented proper
submersion with boundary. Our pushforward is defined so that 
\[
f_{!}^{\kappa}\left(f^{*}\alpha\wedge\beta\right)=\alpha\wedge f_{!}^{\kappa}\beta
\]
for all $\alpha\in\Omega\left(Y;\underline{\cc}\right)$ and $\beta\in\Omega\left(X;\kk\right)$.

\medskip{}
Roughly speaking, this means we write the differentials corresponding
to coordinates of the fiber \emph{after }the differentials pulled
back from the base before integrating them out. This is in agreement
with the definition on page 61 of \cite{bott+tu}.

With this convention we have the following properties. We omit the
proofs, since they are obtained by nothing more than careful book-keeping
of conventions from well-known proofs.
\begin{lem}
\label{lem:Stokes'-Theorem.}(Stokes' Theorem.) Let $\left(f,\kappa\right)$
be an oriented proper submersion with boundary. Use ${\kappa:Or\left(TX\right)\otimes f^{*}Or\left(TY\right)^{\vee}\simeq\kk^{\vee}\otimes f^{*}\lc}$
and $\S$\ref{sub:boundary orientation} to define an isomorphism
${\partial\kappa:Or\left(T\partial X\right)\otimes f|_{\partial X}^{*}Or\left(TY\right)^{\vee}\simeq\left(\kk|_{\partial X}\right)^{\vee}\otimes f|_{\partial X}^{*}\lc}$.
We then have
\end{lem}
\[
\left(f_{!}^{\kappa}d-df_{!}^{\kappa}\right)\omega=\left(-1\right)^{\deg\omega+\deg f_{!}}\left(f|_{\partial X}\right)_{!}^{\partial\kappa}
\]
Henceforth $\deg f_{!}=\dim Y-\dim X$ denotes the degree shift of
the map $f_{!}$.

\begin{lem}
\label{lem:(Projection-Formula.)}(Projection Formula.) Let 
\[
\left(f:X\to Y,\kappa:Or\left(TX\right)\otimes f^{*}Or\left(TY\right)^{\vee}\simeq\kk^{\vee}\otimes f^{*}\lc\right)
\]
be an oriented proper submersion with boundary. Let $\lc_{1},\lc_{2}$
be two additional local systems on $Y$. Set $\kk_{+}:=f^{*}\lc_{1}\otimes\kk\otimes f^{*}\lc_{2}$,
$ $$\lc_{+}:=\lc_{1}\otimes\lc\otimes\lc_{2}$. Denote by ${\kappa_{+}:Or\left(TX\right)\otimes f^{*}Or\left(TY\right)^{\vee}\simeq\kk_{+}^{\vee}\otimes f^{*}\lc_{+}}$
the obvious isomorphism of local systems induced from $\kappa$. Then
for $\alpha_{1}\in\Omega\left(Y;\lc_{1}\right),\;\beta\in\Omega\left(X;\kk\right),\;\mbox{and }\alpha_{2}\in\Omega\left(Y;\lc_{2}\right)$
we have 
\[
f_{!}^{\kappa_{+}}\left(\left(f^{*}\alpha_{1}\right)\wedge\beta\wedge\left(f^{*}\alpha_{2}\right)\right)=\left(-1\right)^{\deg\alpha_{2}\cdot\deg f_{!}}\alpha_{1}\wedge\left(f_{!}^{\kappa}\beta\right)\wedge\alpha_{2}
\]

\end{lem}

The following two lemmas will not be used, we state them for completeness
and future reference.
\begin{lem}
(Composition.) Let 
\[
\left(f:X\to Y,\kappa:Or\left(TX\right)\otimes f^{*}Or\left(TY\right){}^{\vee}\to\kk_{X}^{\vee}\otimes f^{*}\kk_{Y}\right)
\]
be an oriented proper submersion with boundary and 
\[
\left(g:Y\to Z,\kappa':Or\left(TY\right)\otimes f^{*}Or\left(TZ\right){}^{\vee}\to\kk_{Y}^{\vee}\otimes g^{*}\kk_{Z}\right)
\]
be an oriented proper submersion.

Denote by 
\[
f^{*}\kappa'\circ\kappa:Or\left(TX\right)\otimes\left(gf\right)^{*}Or\left(TZ\right)\to\kk_{X}^{\vee}\otimes\left(gf\right)^{*}\kk_{Z}
\]
the local system obtained from combining $\kappa$ and $f^{*}\kappa'$
in the obvious way. We then have 
\[
\left(gf\right)_{!}^{f^{*}\kappa'\circ\kappa}=\left(-1\right)^{\deg f_{!}\deg g_{!}}g_{!}^{\kappa'}\circ f_{!}^{\kappa}
\]

\end{lem}

\begin{lem}
(Push-pull property.) Let 
\[
\xymatrix{X'\ar[r]^{u'}\ar[d]_{f'} & X\ar[d]^{f}\\
Y'\ar[r]_{u} & Y
}
\]
be a Cartesian square, where $\left(f,\kappa\right)$ is an oriented
proper submersion with boundary. It follows that $\partial X'=Y'\times_{Y}\partial X$,
and $f'$ is also an oriented proper submersion with boundary; indeed,
since the fibers of $f'$ are canonically isomorphic to the fibers
of $f$ we find that there's a canonical isomorphism of local systems
\begin{equation}
Or\left(TX'\right)\otimes f'^{*}Or\left(TY'\right)^{\vee}\simeq u'^{*}\left(Or\left(TX\right)\otimes f^{*}Or\left(TY\right)^{\vee}\right),\label{eq:fiber orientation isomorphism}
\end{equation}
and we denote by 
\[
\kappa^{u}:Or\left(TX'\right)\otimes f'^{*}Or\left(TY'\right)^{\vee}\to u'^{*}\kk^{\vee}\otimes f'^{*}\left(u^{*}\lc\right)
\]
the local system isomorphism obtained from combining (\ref{eq:fiber orientation isomorphism})
and $\kappa$ in the obvious way. 
\[
u^{*}\circ f_{!}^{\kappa}=\left(f'\right){}_{!}^{\kappa^{u}}\circ\left(u'\right){}^{*}.
\]

\end{lem}
\bibliographystyle{amsabbrvc}
\bibliography{localization}

\end{document}